\documentclass{article}%
\usepackage{amsfonts}
\usepackage{graphicx}
\usepackage{amsmath}
\usepackage{amssymb}%
\providecommand{\U}[1]{\protect\rule{.1in}{.1in}}
\newtheorem{theorem}{Theorem}

\newtheorem{lemma}[theorem]{Lemma}

\newenvironment{proof}[1][Proof]{\textbf{#1.} }{\ \rule{0.5em}{0.5em}}
\begin{document}

\title{A Spectral Method for Parabolic Differential Equations}
\author{Kendall Atkinson\\Departments of Mathematics \& Computer Science \\The University of Iowa
\and Olaf Hansen\\Department of Mathematics \\California State University - San Marcos
\and David Chien\\Department of Mathematics \\California State University - San Marcos}
\maketitle

\begin{abstract}
We present a spectral method for parabolic partial differential equations with
zero Dirichlet boundary conditions. The region $\Omega$ for the problem is
assumed to be simply-connected and bounded, and its boundary is assumed to be
a smooth surface. An error analysis is given, showing that spectral
convergence is obtained for sufficiently smooth solution functions. Numerical
examples are given in both $\mathbb{R}^{2}$ and $\mathbb{R}^{3}$. 

\end{abstract}

\section{INTRODUCTION\label{sec_intro}}

Consider solving the parabolic partial differential equation%
\begin{equation}
\frac{\partial u\left(  s,t\right)  }{\partial t}=\sum_{k,\ell=1}^{d}%
\frac{\partial}{\partial s_{k}}\left(  a_{k,\ell}(s,t,u\left(  s,t\right)
)\frac{\partial u(s,t)}{\partial s_{\ell}}\right)  +f\left(  s,t,u\left(
s,t\right)  \right)  ,\quad\quad\label{e1}%
\end{equation}
for $s\in\Omega\subseteq\mathbb{R}^{d}$, $0<t\leq T$. The solution $u$ is
subject to the Dirichlet boundary condition
\begin{equation}
u(s,t)\equiv0,\quad\quad s\in\partial\Omega,\quad0<t\leq T\label{e2}%
\end{equation}
and to the initial condition%
\begin{equation}
u\left(  s,0\right)  =u_{0}\left(  s\right)  ,\quad\quad s\in\Omega.\label{e3}%
\end{equation}
The region $\Omega$ is open, bounded, and simply connected in $\mathbb{R}^{d}
$ for some $d\geq2$, and the boundary $\partial\Omega$ is assumed to be
several times continuously differentiable. This paper presents a spectral
method for solving this problem. The functions $a_{i,j}\left(  s,t,z\right)  $
and $f\left(  s,t,z\right)  $ are assumed to be continuous for $\left(
s,t,z\right)  \in\overline{\Omega}\times\left[  0,T\right]  \times\mathbb{R}
$. Additional assumptions are given later in the paper.\ These assumptions are
stronger than needed for the results we obtain, but they simplify the
presentation. In addition, we assume that there is a unique solution $u$ to
the problem (\ref{e1})-(\ref{e3}). For an introduction to the theory of
nonlinear parabolic problems using variational methods, see \cite[Chap.
30]{zeidler}.

We transform the above problem to one over the unit ball $\mathbb{B}_{d}$ in
$\mathbb{R}^{d}$, and then we use Galerkin's method with a suitably chosen
polynomial basis to approximate the solution $u$. This is similar in spirit to
earlier work in \cite{ach2008}, \cite{ah2010}, \cite{ahc2009}. This approach
reduces the problem to the solution of an inital value problem for a system of
ordinary differential equations, for which there is much excellent software.
The convergence analysis of the paper depends on the landmark paper of Douglas
and Dupont \cite{DD1970}. The methods of this paper also extend to having the
functions $a_{i,j}$ and $f$ depend on the first derivatives $\partial
u/\partial s_{j}$, although this is not considered here. For related books on
spectral methods for partial differential equations, see \cite{boyd}%
-\cite{cqhz2}, \cite{GottOrs}, \cite{guo1998}, \cite{ShenTang},
\cite{ShenTangWang}.

The spectral method is presented and analyzed in \S 2, implementation issues
are discussed in \S 3, and numerical examples in $\mathbb{R}^{2}$ and
$\mathbb{R}^{3}$ are given in \S 4.

\section{A spectral method\label{sec_method}}

We transform the problem (\ref{e1})-(\ref{e3}) to one over the unit ball
$\mathbb{B}_{d}$, and then we apply Galerkin's method using multivariate
polynomials as approximations of the solution. To transform a problem defined
on $\Omega$ to an equivalent problem defined on $\mathbb{B}_{d}$, we review
some \ ideas from \cite{ach2008} and \cite{ahc2009}, modifying them as
appropriate for this paper.

Assume the existence of a function%
\begin{equation}
\Phi:\overline{\mathbb{B}}_{d}\underset{onto}{\overset{1-1}{\longrightarrow}%
}\overline{\Omega}\label{e4}%
\end{equation}
with $\Phi$ a twice--differentiable mapping, and let $\Psi=\Phi^{-1}%
:\overline{\Omega}\underset{onto}{\overset{1-1}{\longrightarrow}}%
\overline{\mathbb{B}}_{d}$. \ For $v\in L^{2}\left(  \Omega\right)  $, let%
\begin{equation}
\widetilde{v}(x)=v\left(  \Phi\left(  x\right)  \right)  ,\quad\quad
x\in\overline{\mathbb{B}}_{d}\subseteq\mathbb{R}^{d}\label{e6}%
\end{equation}
and conversely,%
\begin{equation}
v(s)=\widetilde{v}\left(  \Psi\left(  s\right)  \right)  ,\quad\quad
s\in\overline{\Omega}\subseteq\mathbb{R}^{d}.\label{e7}%
\end{equation}
Assuming $v\in H^{1}\left(  \Omega\right)  $, we can show%
\[
\nabla_{x}\widetilde{v}\left(  x\right)  =J\left(  x\right)  ^{\text{T}}%
\nabla_{s}v\left(  s\right)  ,\quad\quad s=\Phi\left(  x\right)
\]
with $J\left(  x\right)  $ the Jacobian matrix for $\Phi$ over the unit ball
$\mathbb{B}_{d}$,%
\begin{equation}
J(x)\equiv\left(  D\Phi\right)  (x)=\left[  \frac{\partial\varphi_{i}%
(x)}{\partial x_{j}}\right]  _{i,j=1}^{d},\quad\quad x\in\overline{\mathbb{B}%
}_{d}.\label{e7b}%
\end{equation}
To use our method for problems over a region $\Omega$, it is necessary to know
explicitly the functions $\Phi$ and $J$. \ We assume%
\begin{equation}
\det J(x)\neq0,\quad\quad x\in\overline{\mathbb{B}}_{d}.\label{e8}%
\end{equation}
Similarly,%
\[
\nabla_{s}v(s)=K(s)^{\text{T}}\nabla_{x}\widetilde{v}(x),\quad\quad x=\Psi(s)
\]
with $K(s)$ the Jacobian matrix for $\Psi$ over $\Omega$. By differentiating
the identity
\[
\Psi\left(  \Phi\left(  x\right)  \right)  =x,\quad\quad x\in\overline
{\mathbb{B}}_{d}%
\]
we obtain%
\[
K\left(  \Phi\left(  x\right)  \right)  =J\left(  x\right)  ^{-1}.
\]
Assumptions about the differentiability of $\widetilde{v}\left(  x\right)  $
can be related back to assumptions on the differentiability of $v(s)$ and
$\Phi(x)$.

\begin{lemma}
\label{transform_lemma}If $\Phi\in C^{m}\left(  \overline{\mathbb{B}}%
_{d}\right)  $ and $v\in C^{k}\left(  \overline{\Omega}\right)  $, then
$\widetilde{v}\in C^{q}\left(  \overline{\mathbb{B}}_{d}\right)  $ with
$q=\min\left\{  k,m\right\}  $.
\end{lemma}

\begin{proof}
A proof is straightforward using (\ref{e6}).$\left.  {}\right.  $%
\hfill\medskip
\end{proof}

\noindent A converse statement can be made as regards $\widetilde{v}$, $v$,
and $\Psi$ in (\ref{e7}).

Often a mapping $\varphi$ is given from $\mathbb{S}^{d-1}$ onto $\partial
\Omega$, and it will not be clear as to how to extend the mapping to $\Phi$
satisfying (\ref{e4}) and (\ref{e8}). This is explored in \cite{ah2011} with
several methods given for constructing $\Phi$.

To obtain a space for approximating the solution $u$ of our problem, we
proceed as follows. Denote by $\Pi_{n}$ the space of polynomials in $d$
variables that are of degree $\leq n$: $p\in\Pi_{n}$ if it has the form%
\[
p(x)=\sum_{\left\vert i\right\vert \leq n}a_{i}x_{1}^{i_{1}}x_{2}^{i_{2}}\dots
x_{d}^{i_{d}}%
\]
with $i$ a multi--integer, $i=\left(  i_{1},\dots,i_{d}\right)  $, and
$\left\vert i\right\vert =i_{1}+\cdots+i_{d}$. Our approximation space with
respect to $\mathbb{B}_{d}$ is%
\begin{equation}
\widetilde{\mathcal{X}}_{n}=\left\{  \left(  1-\left\vert x\right\vert
^{2}\right)  p(x)\mid p\in\Pi_{n}\right\}  \subseteq H_{0}^{1}\left(
\mathbb{B}_{d}\right) \label{e9a}%
\end{equation}
With respect to $\Omega$, the approximating subspace is%
\begin{equation}
\mathcal{X}_{n}=\left\{  \psi\left(  s\right)  =\widetilde{\psi}\left(
\Psi\left(  s\right)  \right)  :\widetilde{\psi}\in\widetilde{\mathcal{X}}%
_{n}\right\}  \subseteq H_{0}^{1}\left(  \Omega\right) \label{e9b}%
\end{equation}
Let $N_{n}=\dim\mathcal{X}_{n}=\dim\widetilde{\mathcal{X}}_{n}=\dim\Pi_{n}$.
For $d=2$, $N_{n}=\left(  n+1\right)  \left(  n+2\right)  /2$.

\subsection{The approximation\label{approximation}}

We reformulate the parabolic problem (\ref{e1})-(\ref{e3}) as a variational
problem. Multiply (\ref{e1}) by an arbitrarily chosen $v\in H_{0}^{1}\left(
\Omega\right)  $ and perform integration by parts, obtaining%
\begin{equation}%
\begin{array}
[c]{r}%
\left(  \dfrac{\partial u\left(  \cdot,t\right)  }{\partial t},v\right)  =-%
{\displaystyle\sum\limits_{i,j=1}^{d}}
{\displaystyle\int_{\Omega}}
a_{i,j}\left(  s,t,u\left(  s,t\right)  \right)  \dfrac{\partial u\left(
s,t\right)  }{\partial s_{i}}\dfrac{\partial v\left(  s,t\right)  }{\partial
s_{j}}\,ds\medskip\\
+\left(  f\left(  \cdot,t,u\left(  \cdot,t\right)  \right)  ,v\right)
,\quad\quad v\in H_{0}^{1}\left(  \Omega\right)  ,\quad t\geq0.
\end{array}
\label{e10}%
\end{equation}
In this equation, $\left(  \cdot,\cdot\right)  $ denotes the usual inner
product for $L^{2}\left(  \Omega\right)  .$ \ \ Equation (\ref{e10}), together
with (\ref{e3}), is used to develop our approximation method.

We look for a solution of the form%
\begin{equation}
u_{n}\left(  s,t\right)  =\sum_{k=1}^{N_{n}}\alpha_{k}\left(  t\right)
\psi_{k}\left(  s\right) \label{e12}%
\end{equation}
with $\left\{  \psi_{1},\dots,\psi_{N}\right\}  $ a basis of $\mathcal{X}_{n}%
$. The coefficients $\left\{  \alpha_{1},\dots,\alpha_{N_{n}}\right\}  $
generally will vary with $n$, but we omit the explicit dependence to simplify
notation. Substitute this $u_{n}$ into (\ref{e10}) and let $v$ run through the
basis elements $\psi_{\ell}$. This results in the following system:%
\begin{equation}%
\begin{array}
[c]{l}%
{\displaystyle\sum\limits_{k=1}^{N_{n}}}
\alpha_{k}^{\prime}\left(  t\right)  \left(  \psi_{k},\psi_{\ell}\right)
\medskip\\
\quad\quad=-%
{\displaystyle\sum\limits_{k=1}^{N_{n}}}
\alpha_{k}\left(  t\right)
{\displaystyle\sum\limits_{i,j=1}^{d}}
{\displaystyle\int_{\Omega}}
a_{i,j}\left(  s,t,%
{\displaystyle\sum\limits_{k=1}^{N_{n}}}
\alpha_{k}\left(  t\right)  \psi_{k}\left(  s\right)  \right)  \dfrac
{\partial\psi_{k}\left(  s,t\right)  }{\partial s_{i}}\dfrac{\partial
\psi_{\ell}\left(  s,t\right)  }{\partial s_{j}}\,ds\medskip\\
\quad\quad\quad\quad+\left(  f\left(  \cdot,t,%
{\displaystyle\sum\limits_{k=1}^{N_{n}}}
\alpha_{k}\left(  t\right)  \psi_{k}\right)  ,\psi_{\ell}\right)  ,\quad
\ell=1,\dots,N_{n},\quad t\geq0
\end{array}
\label{e14}%
\end{equation}
This is a system of ordinary differential equations for the coefficients
$\alpha_{k}$, for $k=1,\dots,N_{n}$. For the initial conditions, calculate%
\begin{equation}
u_{0}\left(  s\right)  \approx u_{0,n}\left(  s\right)  \equiv\sum
_{k=1}^{N_{n}}\alpha_{k}^{(0)}\psi_{k}\left(  s\right) \label{e16}%
\end{equation}
by some means, and then use
\begin{equation}
\alpha_{k}\left(  0\right)  =\alpha_{k}^{(0)},\quad\quad k=1,\dots
,N_{n}.\label{e18}%
\end{equation}
The implementation of (\ref{e12})-(\ref{e18}) is discussed in
\S \ref{sec_implement}.

\subsection{Convergence analysis\label{con_analysis}}

Our error analysis of (\ref{e12})-(\ref{e18}) is based on Douglas and Dupont
\cite[Thm. 7.1]{DD1970}; and as in that paper, we assume the functions
$\left\{  a_{i,j}\right\}  $ and $f$ satisfy a number of properties.

\begin{enumerate}
\item[\textbf{A1}] As stated earlier, we assume the functions $a_{i,j}\left(
s,t,z\right)  $ and $f\left(  s,t,z\right)  $ are continuous for $\left(
s,t,z\right)  \in\overline{\Omega}\times\left[  0,T\right]  \times\mathbb{R}
$. Moreover, assume%
\[
\left\vert f\left(  s,t,r\right)  -f\left(  s,t,\rho\right)  \right\vert \leq
K\left\vert r-\rho\right\vert ,
\]
for all $\left(  s,t,r\right)  ,\left(  s,t,\rho\right)  \in\overline{\Omega
}\times\left[  0,T\right]  \times\mathbb{R}$, and%
\[
\left\vert a_{i,j}\left(  s,t,r\right)  -a_{i,j}\left(  s,t,\rho\right)
\right\vert \leq K\left\vert r-\rho\right\vert
\]
for all $\left(  s,t,r\right)  ,\left(  s,t,\rho\right)  \in\overline{\Omega
}\times\left[  0,T\right]  \times\mathbb{R}$, $1\leq i,j\leq d$.

\item[\textbf{A2}] We assume that the matrix $A\left(  s,t,z\right)
\equiv\left[  a_{i,j}\left(  s,t,z\right)  \right]  _{i,j=1}^{d}$ is
symmetric, positive definite, and has a spectrum that is bounded above and
below by positive constants $\eta_{1}$ and $\eta_{2}$, uniformly so for
$\left(  s,t,z\right)  \in\overline{\Omega}\times\left[  0,T\right]
\times\mathbb{R} $.
\end{enumerate}

\begin{theorem}
\label{DD_Thm}(Douglas and Dupont) Assume the functions $a_{i,j}\left(
s,t,z\right)  $ and $f\left(  s,t,z\right)  $ satisfy the conditions
\textbf{A1}-\textbf{A2}. Let $u$ be the solution of (\ref{e1})-(\ref{e3}) and
assume it is continuously differentiable over $\overline{\Omega}\times\left[
0,T\right]  $. Let $u_{n}$ be the solution of (\ref{e12})-(\ref{e18}). Then
there are positive constants $\gamma$ and $C$ for which%
\begin{equation}%
\begin{array}
[c]{l}%
\Vert u-u_{n}\Vert_{L^{2}\times L^{\infty}}^{2}+\gamma\Vert u-u_{n}%
\Vert_{H_{0}^{1}\times L^{2}}^{2}\smallskip\\
\quad\leq C\,\left\{  \Vert u_{0}-u_{0,n}\Vert_{L^{2}}^{2}+\Vert
u-w\Vert_{L^{2}\times L^{\infty}}^{2}\right. \\
\quad\quad\quad\left.  +\Vert u-w\Vert_{H_{0}^{1}\times L^{2}}^{2}+\Vert
\frac{\partial}{\partial t}\left(  u-w\right)  \Vert_{L^{2}\times L^{2}}%
^{2}\right\}
\end{array}
\label{e24}%
\end{equation}
for any $w$ of the form given on the right side of (\ref{e12}).
\end{theorem}

The norms used in (\ref{e24}) are given by%
\begin{align*}
\Vert v\Vert_{L^{2}\times L^{\infty}}  &  =\sup_{0\leq t\leq T}\Vert v\left(
\cdot,t\right)  \Vert_{L^{2}\left(  \Omega\right)  }\medskip\\
\Vert v\Vert_{L^{2}\times L^{2}}  &  =\Vert v\Vert_{L^{2}\left(  \Omega
\times\left[  0,T\right]  \right)  }\medskip\\
\Vert v\Vert_{H_{0}^{1}\times L^{2}}^{2}  &  =\int_{0}^{T}\Vert v\left(
\cdot,t\right)  \Vert_{H_{0}^{1}\left(  \Omega\right)  }^{2}\,dt
\end{align*}
The assumptions of the theorem imply the assumptions used in \cite[Thm.
7.1]{DD1970}, and the conclusion follows from the cited paper.

To apply this theorem, we need bounds on the norms given in (\ref{e24}) for
$u-w$. To obtain these, we use the following approximation theoretic result
that follows from Ragozin \cite{ragozin}.

\begin{lemma}
\label{ragozin_extension}Assume that $g\left(  x,t\right)  ,\,\partial
g\left(  x,t\right)  /\partial t\ $are $k$ times continously differentiable
with respect to $x\in\overline{\mathbb{B}}_{d}$, \ for some $k\geq0$ and
$0\leq t\leq T$. Further, assume \ that all such $k^{\text{th}}$-order
derivatives satisfy a H\"{o}lder condition with exponent $\gamma\in(0,1]$
\ and with respect to $x\in\overline{\mathbb{B}}_{d}$,%
\[%
\begin{tabular}
[c]{c}%
$\left\vert h\left(  x,t\right)  -h\left(  y,t\right)  \right\vert \leq
c_{k,\gamma}\left(  g\right)  \left\vert x-y\right\vert ^{\gamma},\medskip$\\
$\left\vert \dfrac{\partial h\left(  x,t\right)  }{\partial t}-\dfrac{\partial
h\left(  y,t\right)  }{\partial t}\right\vert \leq c_{k,\gamma}\left(
g\right)  \left\vert x-y\right\vert ^{\gamma},$%
\end{tabular}
\]
uniformly for $x,y\in\overline{\mathbb{B}}_{d}$ and $0\leq t\leq T$, where $h
$ denotes a generic $k^{\text{th}}$-order derivative of $g$ with respect to
$x\in\overline{\mathbb{B}}_{d}$. The quantity $c_{k,\gamma}\left(  g\right)  $
is called the H\"{o}lder constant. \ Let \ $\left\{  \varphi_{1},\dots
,\varphi_{N}\right\}  $ denote a basis of $\Pi_{n}$. Then for each degree
$n\geq1$, there exists
\[
g_{n}\left(  x,t\right)  =\sum_{k=1}^{N_{n}}\beta_{k}\left(  t\right)
\varphi_{k}\left(  x\right)
\]
which satisfies%
\[
\max_{0\leq t\leq T}\max_{x\in\overline{\mathbb{B}}_{d}}\left\vert g\left(
x,t\right)  -g_{n}\left(  x,t\right)  \right\vert \leq\dfrac{b_{k,\gamma}%
}{n^{k+\gamma}}c_{k,\gamma}\left(  g\right)  ,
\]%
\[
\max_{0\leq t\leq T}\max_{x\in\overline{\mathbb{B}}_{d}}\left\vert
\dfrac{\partial g\left(  x,t\right)  }{\partial t}-\dfrac{\partial
g_{n}\left(  x,t\right)  }{\partial t}\right\vert \leq\dfrac{b_{k,\gamma}%
}{n^{k+\gamma}}c_{k,\gamma}\left(  g\right)  ,
\]
for some constant $b_{k,\gamma}>0$ that is independent of $g$.
\end{lemma}

\begin{proof}
This result can be obtained by a careful examination of the proof of Ragozin
\cite[Thm. 3.4]{ragozin}. A similar argument for approximation of a
parameterized family $g\left(  x,t\right)  $ over the unit sphere
$\mathbb{S}^{d-1}$ is given in \cite{AtkHan}. The present result over
$\mathbb{B}_{d}$ follows by combining that of \cite[\S 4.2.5]{AtkHan} over
$\mathbb{S}^{d}$ with the argument of Ragozin over $\mathbb{B}_{d}$.$\medskip
$\hfill
\end{proof}

Next, we must look at the approximation of the solution $\widetilde{u}\left(
x,t\right)  $ by means of polynomials of the form given on the right side of
(\ref{e12}). To do this, we use a trick from \cite[(9)-(15)]{ach2008}. Begin
with the result that%
\begin{equation}
\Delta:\widetilde{\mathcal{X}}_{n}%
\underset{onto}{\overset{1-1}{\longrightarrow}}\Pi_{n}\text{.}\label{e28}%
\end{equation}
A short proof is given in \cite[\S 2.2]{ah2005}. For any $t\in\left[
0,T\right]  $, consider a function $\widetilde{u}$ which satisfies
$\widetilde{u}\left(  x,t\right)  =0$ for all $x\in\mathbb{S}^{d-1}%
=\partial\mathbb{B}_{d}$. Define $g=\Delta_{x}\widetilde{u}$. Then%
\[
\widetilde{u}\left(  x,t\right)  =\int_{\mathbb{B}_{d}}G\left(  x,y\right)
g\left(  y,t\right)  \,dy,\quad\quad x\in\overline{\mathbb{B}}_{d},
\]
with $G$ the Green's function for the elliptic boundary value problem%
\begin{align*}
-\Delta v\left(  x\right)   &  =g\left(  x\right)  ,\quad\quad x\in
\mathbb{B}_{d},\\
v\left(  x\right)   &  =0,\quad\quad x\in\mathbb{S}^{d-1}.
\end{align*}
For example, in $\mathbb{R}^{2}$,
\[
G\left(  x,y\right)  =\dfrac{1}{2\pi}\log\dfrac{\left\vert x-y\right\vert
}{\left\vert \mathcal{T}(x)-y\right\vert },\quad\quad x,y\in\mathbb{B}_{2},
\]
with $\mathcal{T}(x)$ the inverse of $x$ with respect to the unit circle
$\mathbb{S}^{1}$. Let $g_{n}\left(  x,t\right)  $ be the polynomial referenced
in the preceding Lemma \ref{ragozin_extension}, and define%
\begin{equation}
\widetilde{w}_{n}\left(  x,t\right)  =\int_{\mathbb{B}_{d}}G\left(
x,y\right)  g_{n}\left(  y,t\right)  \,dy,\quad\quad x\in\overline{\mathbb{B}%
}_{d}.\label{e30}%
\end{equation}
From (\ref{e28}), $\widetilde{w}_{n}\left(  \cdot,t\right)  \in
\widetilde{\mathcal{X}}_{n}\,$, $0\leq t\leq T$; and $\widetilde{w}_{n}$ is an
approximation of the original function $\widetilde{u}$.

\begin{lemma}
\label{approx_lemma}Assume $\widetilde{u}\left(  \cdot,t\right)  \in
C^{k,\gamma}\left(  \overline{\mathbb{B}}_{d}\right)  $ for $0\leq t\leq T$,
with $k\geq2$, $0<\gamma\leq1$. Then for $n\geq1$, the function $\widetilde{w}%
_{n}\left(  x,t\right)  $ of (\ref{e30}) is of the form%
\begin{equation}
\widetilde{w}_{n}\left(  x,t\right)  =\sum_{k=1}^{N_{n}}\alpha_{k}\left(
t\right)  \widetilde{\psi}_{k}\left(  x\right) \label{e31}%
\end{equation}
and it satisfies%
\begin{align}
\left\Vert \widetilde{u}\left(  \cdot,t\right)  -\widetilde{w}_{n}\left(
\cdot,t\right)  \right\Vert _{C\left(  \overline{\mathbb{B}}_{d}\right)  } &
\leq\dfrac{b_{k,\gamma}\alpha_{1}\left(  G\right)  }{n^{k+\gamma-2}%
}c_{k,\gamma}\left(  g\right)  ,\medskip\label{e32}\\
\left\Vert \nabla_{x}\left[  \widetilde{u}\left(  \cdot,t\right)
-\widetilde{w}_{n}\left(  \cdot,t\right)  \right]  \right\Vert _{C\left(
\overline{\mathbb{B}}_{d}\right)  } &  \leq\dfrac{b_{k,\gamma}\alpha
_{2}\left(  G\right)  }{n^{k+\gamma-2}}c_{k,\gamma}\left(  g\right)
,\medskip\label{e33}\\
\left\Vert \frac{\partial}{\partial t}\left[  \widetilde{u}\left(
\cdot,t\right)  -\widetilde{w}_{n}\left(  \cdot,t\right)  \right]  \right\Vert
_{C\left(  \overline{\mathbb{B}}_{d}\right)  } &  \leq\dfrac{b_{k,\gamma
}\alpha_{1}\left(  G\right)  }{n^{k+\gamma-2}}c_{k,\gamma}\left(  g\right)
\label{e34}%
\end{align}
for $0\leq t\leq T$. The constants $\alpha_{1}$ and $\alpha_{2}$ are given by%
\begin{align*}
\alpha_{1}\left(  G\right)   &  =\max_{x\in\overline{\mathbb{B}}_{d}}%
\int_{\mathbb{B}_{d}}\left\vert G\left(  x,y\right)  \right\vert
\,dy,\medskip\\
\alpha_{2}\left(  G\right)   &  =\max_{x\in\overline{\mathbb{B}}_{d}}%
\int_{\mathbb{B}_{d}}\left\vert \nabla_{x}G\left(  x,y\right)  \right\vert
\,dy,
\end{align*}
and these are easily shown to be finite. The remaining constants $b_{k,\gamma
}$ and $c_{k,\gamma}\left(  g\right)  $ are taken from Lemma
\ref{ragozin_extension}.
\end{lemma}

\begin{proof}
For the error in approximating $\widetilde{u}$, we have%
\begin{align*}
\widetilde{u}\left(  x,t\right)  -\widetilde{w}_{n}\left(  x,t\right)   &
=\int_{\mathbb{B}_{d}}G\left(  x,y\right)  \left[  g\left(  y,t\right)
-g_{n}\left(  y,t\right)  \right]  \,dy,\medskip\\
\nabla_{x}\left[  \widetilde{u}\left(  \cdot,t\right)  -\widetilde{w}%
_{n}\left(  \cdot,t\right)  \right]   &  =\int_{\mathbb{B}_{d}}\nabla
_{x}G\left(  x,y\right)  \left[  g\left(  y,t\right)  -g_{n}\left(
y,t\right)  \right]  \,dy,\medskip\\
\frac{\partial}{\partial t}\left[  \widetilde{u}\left(  \cdot,t\right)
-\widetilde{w}_{n}\left(  \cdot,t\right)  \right]   &  =\int_{\mathbb{B}_{d}%
}G\left(  x,y\right)  \frac{\partial}{\partial t}\left[  g\left(  y,t\right)
-g_{n}\left(  y,t\right)  \right]  \,dy
\end{align*}
Thus%
\[
\left\Vert \widetilde{u}\left(  \cdot,t\right)  -\widetilde{w}_{n}\left(
\cdot,t\right)  \right\Vert _{C\left(  \overline{\mathbb{B}}_{d}\right)  }%
\leq\alpha_{1}\left(  G\right)  \left\Vert g\left(  \cdot,t\right)
-g_{n}\left(  \cdot,t\right)  \right\Vert _{C\left(  \overline{\mathbb{B}}%
_{d}\right)  }%
\]
showing (\ref{e32}); and (\ref{e33}) and (\ref{e34}) follow similarly.$\left.
{}\right.  $\hfill\medskip
\end{proof}

These results can be extended to the approximation of $u\left(  \cdot
,t\right)  $ over $\Omega$, by the subspace $\mathcal{X}_{n}$.

\begin{lemma}
Assume $u\left(  \cdot,t\right)  \in C^{k,\gamma}\left(  \overline{\Omega
}\right)  $ for $0\leq t\leq T$, with $k\geq2$, $0<\gamma\leq1$; and assume
$\Phi\in C^{m}\left(  \overline{\mathbb{B}}_{d}\right)  $ with $m\geq k+3$.
Then for $n\geq1$ there exists
\begin{equation}
w_{n}\left(  s,t\right)  =\sum_{k=1}^{N_{n}}\alpha_{k}\left(  t\right)
\psi_{k}\left(  s\right)  ,\quad\quad s\in\overline{\Omega},\quad0\leq t\leq
T,\label{e40}%
\end{equation}
for which
\begin{align}
\left\Vert u\left(  \cdot,t\right)  -w_{n}\left(  \cdot,t\right)  \right\Vert
_{C\left(  \overline{\Omega}\right)  } &  \leq\dfrac{\omega_{1}\left(
k,\gamma,u\right)  }{n^{k+\gamma-2}},\medskip\label{e41}\\
\left\Vert \nabla_{x}\left[  u\left(  \cdot,t\right)  -w_{n}\left(
\cdot,t\right)  \right]  \right\Vert _{C\left(  \overline{\Omega}\right)  } &
\leq\dfrac{\omega_{2}\left(  k,\gamma,u\right)  }{n^{k+\gamma-2}}%
,\medskip\label{e42}\\
\left\Vert \frac{\partial}{\partial t}\left[  u\left(  \cdot,t\right)
-w_{n}\left(  \cdot,t\right)  \right]  \right\Vert _{C\left(  \overline
{\Omega}\right)  } &  \leq\dfrac{\omega_{3}\left(  k,\gamma,u\right)
}{n^{k+\gamma-2}}\label{e43}%
\end{align}
for $0\leq t\leq T$.
\end{lemma}

\begin{proof}
Use the transformation $s=\Phi\left(  x\right)  $ to move between functions
over $\Omega$ and functions over $\mathbb{B}_{d}$. By means Lemma
\ref{transform_lemma} for the transformation $\Phi$, these results follow
immediately from Lemma \ref{approx_lemma}.$\left.  {}\right.  $\hfill\medskip
\end{proof}

Combining these results with the Douglas and Dupont Theorem \ref{DD_Thm} leads
to the following convergence result for the Galerkin method (\ref{e14}%
)-(\ref{e18}).

\begin{theorem}
Assume that the solution $u$ of the parabolic problem (\ref{e1})-(\ref{e3})
satisfies $u\left(  \cdot,t\right)  \in C^{k,\gamma}\left(  \overline{\Omega
}\right)  $ for $0\leq t\leq T$, with $k\geq2$, $0<\gamma\leq1$. Moreover,
assume the transformation $\Phi\in C^{m}\left(  \overline{\mathbb{B}}%
_{d}\right)  $ with $m\geq k+3$. Then for $n\geq1$, the solution $u_{n}$ of
(\ref{e14})-(\ref{e18}) satisfies%
\[
\Vert u-u_{n}\Vert_{L^{2}\times L^{\infty}}^{2},\ \Vert u-u_{n}\Vert
_{H_{0}^{1}\times L^{2}}^{2}=\mathcal{O}\left(  n^{-\left(  k+\gamma-2\right)
}\right)  .
\]

\end{theorem}

\section{Implementation issues\label{sec_implement}}

Recall the method (\ref{e12})-(\ref{e18}) and the notation used there. For
notation, let%
\[
\mathsf{a}_{N}\left(  t\right)  =\left[  \alpha_{1}\left(  t\right)
,\dots,\alpha_{N}\left(  t\right)  \right]  ^{\text{T}}.
\]
The system (\ref{e14}) can be written symbolically as%
\begin{equation}
G_{n}\mathsf{a}_{N}^{^{\prime}}\left(  t\right)  =B_{n}\left(  t,u_{n}\right)
\mathsf{a}_{N}\left(  t\right)  +\boldsymbol{f}_{N}\left(  t,u_{n}\right)
,\label{e50}%
\end{equation}%
\begin{equation}
G_{n}=\left[  \left(  \psi_{k},\psi_{\ell}\right)  \right]  _{k,\ell=1}%
^{N},\label{e52}%
\end{equation}%
\begin{equation}
\left(  B_{n}\left(  t,u_{n}\right)  \right)  _{k,\ell}=-\sum_{i,j=1}^{d}%
\int_{\Omega}a_{i,j}\left(  s,t,u_{n}\left(  s,t\right)  \right)
\dfrac{\partial\psi_{k}\left(  s,t\right)  }{\partial s_{i}}\dfrac
{\partial\psi_{\ell}\left(  s,t\right)  }{\partial s_{j}}\,ds,\label{e54}%
\end{equation}%
\begin{equation}
\boldsymbol{f}_{N}\left(  t,u_{n}\right)  _{\ell}=\left(  f\left(
\cdot,t,u_{n}\left(  \cdot,t\right)  \right)  ,\psi_{\ell}\right)  ,\quad
\quad\ell=1,\dots,N.\label{e56}%
\end{equation}

For the implementation, we discuss separately the cases of $\Omega
\subseteq\mathbb{R}^{2}$ and $\Omega\subseteq\mathbb{R}^{3}$. In both cases we
must address the following issues

\begin{enumerate}
\item[\textbf{A1}.] Select a basis $\left\{  \psi_{1},\dots,\psi_{N}\right\}
$ for $\mathcal{X}_{n}$.

\item[\textbf{A2}.] Discuss the numerical integration of the integrals in
(\ref{e52})-(\ref{e56}). \ 

\item[\textbf{A3}.] Approximate the initial value $u_{0}$ by some $u_{0,n}%
\in\mathcal{X}_{n}$, as suggested in (\ref{e16}).

\item[\textbf{A4}.] Discuss the solution of the nonlinear system of
differential equations (\ref{e50}).

\item[\textbf{A5}.] Evaluate the solution $u_{n}$ at points of $\Omega$ for
each given $t$.
\end{enumerate}

Several of these issues were addressed in the previous papers \cite{ach2008},
\cite{ah2010}, \cite{ahc2009}, and we refer to the discussion in those papers
for more complete discussions.

\subsection{Two dimensions\label{sec_2D}}

Let $\Pi_{n}\left(  \mathbb{B}_{2}\right)  $ denote the restriction to
$\mathbb{B}_{2}$ of the polynomials over $\mathbb{R}^{2}$. To construct a
basis for the approximation space $\mathcal{X}_{n}$ of (\ref{e9b}), begin by
choosing an orthonormal basis $\left\{  \varphi_{1},\dots,\varphi_{N}\right\}
$ for $\Pi_{n}\left(  \mathbb{B}_{2}\right)  $, using the standard inner
product for $L^{2}\left(  \mathbb{B}_{2}\right)  $. \ The dimension of
$\Pi_{n}\left(  \mathbb{B}_{2}\right)  $ is
\[
N\equiv N_{n}=\frac{1}{2}\left(  n+1\right)  \left(  n+2\right)
\]
There are many possible choices of an orthonormal basis, a number of which are
enumerated in \cite[\S 2.3.2]{DX} and \cite[\S 1.2]{xu2004}. We have chosen
one that is particularly convenient for our computations. These are the `ridge
polynomials' introduced by Logan and\ Shepp \cite{Loga} for solving an image
reconstruction problem. We summarize here the results needed for our work.

Let
\[
\mathcal{V}_{n}=\left\{  P\in\Pi_{n}\left(  \mathbb{B}_{2}\right)  :\left(
P,Q\right)  =0\quad\forall Q\in\Pi_{n-1}\right\}
\]
the polynomials of degree $n$ that are orthogonal to all elements of
$\Pi_{n-1}\left(  \mathbb{B}_{2}\right)  $. Then the dimension of
$\mathcal{V}_{n}$ is $n+1$; moreover,%
\begin{equation}
\Pi_{n}\left(  \mathbb{B}_{2}\right)  =\mathcal{V}_{0}\oplus\mathcal{V}%
_{1}\oplus\cdots\oplus\mathcal{V}_{n}\label{e100}%
\end{equation}
It is standard to construct orthonormal bases of each $\mathcal{V}_{n}$ and to
then combine them to form an orthonormal basis of $\Pi_{n}\left(
\mathbb{B}_{2}\right)  $ using the latter decomposition. \ As an orthonormal
basis of $\mathcal{V}_{n}$ we use%
\begin{equation}
\widetilde{\varphi}_{n,k}(x)=\frac{1}{\sqrt{\pi}}U_{n}\left(  x_{1}\cos\left(
kh\right)  +x_{2}\sin\left(  kh\right)  \right)  ,\quad x\in D,\quad
h=\frac{\pi}{n+1}\label{e101}%
\end{equation}
for $k=0,1,\dots,n$. The function $U_{n}$ is the Chebyshev polynomial of the
second kind of degree $n$:%
\[
U_{n}(t)=\frac{\sin\left(  n+1\right)  \theta}{\sin\theta},\quad\quad
t=\cos\theta,\quad-1\leq t\leq1,\quad n=0,1,\dots
\]
The family $\left\{  \widetilde{\varphi}_{n,k}\right\}  _{k=0}^{n}$ is an
orthonormal basis of $\mathcal{V}_{n}$. As a basis of $\Pi_{n}$, we order
$\left\{  \widetilde{\varphi}_{n,k}\right\}  $ lexicographically based on the
ordering in (\ref{e101}) and (\ref{e100}):%
\[
\left\{  \widetilde{\varphi}_{\ell}\right\}  _{\ell=1}^{N}\equiv\left\{
\widetilde{\varphi}_{0,0},\,\widetilde{\varphi}_{1,0},\,\widetilde{\varphi
}_{1,1},\,\widetilde{\varphi}_{2,0},\,\dots,\,\widetilde{\varphi}%
_{n,0},\,\dots,\widetilde{\varphi}_{n,n}\right\}
\]

Returning to (\ref{e9b}), we define%
\[
\widetilde{\psi}_{n,k}(x)=\left(  1-\left\vert x\right\vert ^{2}\right)
\widetilde{\varphi}_{n,k}(x)
\]
and the basis $\left\{  \psi_{m,k}:0\leq k\leq m,\ 0\leq m\leq n\right\}  $
for $\mathcal{X}_{n}$ is defined using (\ref{e9b}),
\[
\psi_{m,k}\left(  s\right)  =\widetilde{\psi}_{n,k}(x),\quad\quad
s=\Phi\left(  x\right)  .
\]
We will also refer to this basis as $\left\{  \psi_{1},\dots,\psi_{N}\right\}
$. In general, this is not an orthonormal basis; but the hope is that
$\left\{  \widetilde{\varphi}_{\ell}\right\}  _{\ell=1}^{N}$ being orthonormal
will result in a reasonably well-conditioned matrix for the linear systems
associated with the solution of (\ref{e14}). Examples of this for elliptic
problems are given in \cite{ach2008}, \cite{ah2010}, \cite{ahc2009}.

To calculate the first order partial derivatives of $\widetilde{\psi}%
_{n,k}(x)$, we need $U_{n}^{\prime}(t)$. The values of $U_{n}(t)$ and
$U_{n}^{^{\prime}}(t)$ are evaluated using the standard triple recursion
relations%
\begin{align*}
U_{n+1}(t)  &  =2tU_{n}(t)-U_{n-1}(t)\\
U_{n+1}^{^{\prime}}(t)  &  =2U_{n}(t)+2tU_{n}^{^{\prime}}(t)-U_{n-1}%
^{^{\prime}}(t)
\end{align*}
Second derivatives, if needed, can be evaluated similarly.

For the integrals in (\ref{e14}), for any dimension $d\geq2$, we first
transform them to integrals over $\mathbb{B}_{d}$. For an arbitrary function
$g$ defined on $\Omega$, use the transformation $s=\Phi\left(  x\right)  $ to
write%
\[
\int_{\Omega}g\left(  s\right)  \,ds=\int_{\mathbb{B}_{d}}g\left(  \Phi\left(
x\right)  \right)  \,\det J\left(  x\right)  \,dx
\]
with $J\left(  x\right)  $ the Jacobian matrix (\ref{e7b}) for $\Phi\left(
x\right)  $. Applying this to the integrals in (\ref{e14}),%
\begin{equation}
\left(  \psi_{k},\psi_{\ell}\right)  =\int_{\Omega}\psi_{k}\left(  s\right)
\psi_{\ell}\left(  s\right)  \,ds=\int_{\mathbb{B}_{d}}\widetilde{\psi}%
_{k}\left(  x\right)  \widetilde{\psi}_{\ell}\left(  x\right)  \,\det J\left(
x\right)  \,dx\label{e110}%
\end{equation}%
\begin{equation}%
\begin{array}
[c]{l}%
\left(  f\left(  \cdot,t,%
{\displaystyle\sum\limits_{k=1}^{N_{n}}}
\alpha_{k}\left(  t\right)  \psi_{k}\right)  ,\psi_{\ell}\right)  \medskip\\
\quad\quad=%
{\displaystyle\int_{\mathbb{B}_{d}}}
f\left(  \Phi\left(  x\right)  ,t,%
{\displaystyle\sum\limits_{k=1}^{N_{n}}}
\alpha_{k}\left(  t\right)  \widetilde{\psi}_{k}\left(  x\right)  \right)
\widetilde{\psi}_{k}\left(  x\right)  \,\det J\left(  x\right)  \,dx
\end{array}
\label{e112}%
\end{equation}%
\begin{equation}%
\begin{array}
[c]{l}%
{\displaystyle\sum\limits_{i,j=1}^{d}}
{\displaystyle\int_{\Omega}}
a_{i,j}\left(  s,t,u_{n}\left(  s,t\right)  \right)  \dfrac{\partial\psi
_{k}\left(  s\right)  }{\partial s_{i}}\dfrac{\partial\psi_{\ell}\left(
s\right)  }{\partial s_{j}}\,ds\medskip\\
\quad\quad=%
{\displaystyle\int_{\Omega}}
\left\{  \nabla\psi_{k}\left(  s\right)  \right\}  ^{\text{T}}A\left(
s,u_{n}\left(  s,t\right)  \right)  \left\{  \nabla\psi_{\ell}\left(
s\right)  \right\}  \,ds\medskip\\
\quad\quad=%
{\displaystyle\int_{\mathbb{B}_{d}}}
\left\{  \nabla\widetilde{\psi}_{k}\left(  x\right)  \right\}  ^{\text{T}%
}\widetilde{A}\left(  x,t,%
{\displaystyle\sum\limits_{k=1}^{N_{n}}}
\alpha_{k}\left(  t\right)  \widetilde{\psi}_{k}\left(  x\right)  \right)
\left\{  \nabla\widetilde{\psi}_{\ell}\left(  x\right)  \right\}  \det
J\left(  x\right)  \,dx
\end{array}
\label{e114}%
\end{equation}
with%
\begin{equation}
\widetilde{A}\left(  x,t,z\right)  =J\left(  x\right)  ^{-1}A\left(
\Phi\left(  x\right)  ,t,z\right)  J\left(  x\right)  ^{-\text{T}%
}.\label{e116}%
\end{equation}

For the numerical approximation of the integrals in (\ref{e110})-(\ref{e114})
with $d=2$, the integrals being evaluated over the unit disk $\mathbb{B}_{2}$,
write a general function $g$ as%
\[
g\left(  x\right)  =\widehat{g}\left(  r,\theta\right)  \equiv g\left(
r\cos\theta,r\sin\theta\right)  .
\]
Then use the formula%
\begin{equation}
\int_{\mathbb{B}_{2}}g(x)\,dx\approx\sum_{l=0}^{q}\sum_{m=0}^{2q}%
\widehat{g}\left(  r_{l},\frac{2\pi\,m}{2q+1}\right)  \omega_{l}\frac{2\pi
}{2q+1}r_{l}\label{e118}%
\end{equation}
with $q\geq1$ an integer. Here the numbers $\omega_{l}$ are the weights of the
$\left(  q+1\right)  $-point Gauss-Legendre quadrature formula on $[0,1]$. The
formula (\ref{e118}) uses the trapezoidal rule with $2q+1$ subdivisions for
the integration over $\mathbb{B}_{2}$ in the azimuthal variable. This
quadrature (\ref{e118}) is exact for all polynomials $g\in\Pi_{2q}\left(
\mathbb{B}_{2}\right)  $.

To approximate the initial condition $u_{0}$, as in (\ref{e16}), \ we
approximate $u_{0}\left(  \Phi\left(  x\right)  \right)  $ by its orthogonal
projection onto $\widetilde{\mathcal{X}}_{n}$,%
\[
\mathcal{P}_{n}\left(  u_{0}\circ\Phi\right)  =\sum_{j=1}^{N_{n}}\beta
_{j}\widetilde{\psi}_{j}%
\]
The coefficients $\left\{  \beta_{j}\right\}  $ are obtained by solving the
linear system%
\begin{align}
\sum_{j=1}^{N_{n}}\beta_{j}\left(  \widetilde{\psi}_{j},\widetilde{\psi}%
_{i}\right)   &  =  \left(  u_{0}\circ\Phi,\widetilde{\psi}_{i}\right)
,\quad\quad i=1,\dots,N_{n}.\label{e1005}%
\end{align}
We approximate further by applying the numerical integration (\ref{e118}) to
each of the inner products in this system. With $q\geq n+2$, the matrix
coefficients for the left side of this linear system will be evaluated
exactly. The result of solving this system with the associated numerical
integration yields an approximation to $u_{0}\left(  \Phi\left(  x\right)
\right)  $; and using $s=\Phi\left(  x\right)  $, we have an initial estimate
of the form given in (\ref{e16}).

To solve the system of ordinary differential equations (\ref{e14}), we have
used the \textsc{Matlab} program \texttt{ode15s}, which is based on the
multistep BDF methods of orders 1 through 5; see \cite[\S 8.2]{AHS}, \cite[p.
60]{SGT}. In general, there is often stiffness when solving differential
equations that arise from using a method of lines approximation for parabolic
problems, and that is our reasoning for using the stiff ode code
\texttt{ode15s} rather than an ordinary Runge-Kutta or multistep code. \ No
difficulty arose in solving any of our examples when using this code, although
further work is needed to know whether or not a stiff ode code is indeed
needed. In our numerical examples, we will give some data on condition numbers
that arise in our method.

\subsection{Three dimensions\label{sec_3D}}

Here we denote by $\Pi_{n}(\mathbb{B}_{3})$ the restriction to $\mathbb{B}%
_{3}$ of polynomials over $\mathbb{R}^{3}$ of degree $n$ or less. The first
difference to the two dimensional case is that the dimension of $\Pi
_{n}(\mathbb{B}_{3})$ is given by
\[
N\equiv N_{n}=\frac{1}{6}(n+1)(n+2)(n+3).
\]
But as with the two dimensional case, there is a wide range of orthonormal
basis functions; see \cite{DX}. We choose the following orthormal basis for
$\Pi_{n}(\mathbb{B}_{3})$
\begin{align}
\widetilde{\varphi}_{m,j,k}(x)  & =c_{m,j}\,p_{j}^{(0,m-2j+\frac{1}{2}%
)}(2|x|^{2}-1)S_{\beta,m-2j}(x)\nonumber\\
& =c_{m,j}\,|x|^{m-2j}\;p_{j}^{(0,m-2j+\frac{1}{2})}(2|x|^{2}-1)S_{\beta
,m-2j}\left(  \frac{x}{|x|}\right) \label{e1003}\\
& j=0,\ldots,\lfloor m/2\rfloor,\;\beta=0,1,\ldots,2(m-2j),\;m=0,\ldots
,n\nonumber
\end{align}
The constants $c_{m,j}=2^{\frac{5}{4}+\frac{m}{2}-j}$ normalize the functions
to length one. The functions $p^{(0,m-2j+\frac{1}{2})}$ are the normalized
Jacobi polynomials on the interval $[-1,1]$ with respect to the inner product
\[
(v,w)=\int_{-1}^{1}(1+t)^{m-2j+\frac{1}{2}}\;v(t)w(t)\;dt
\]
Finally the functions $S_{\beta,m-2j}$ are spherical harmonic functions given
by
\[
S_{\beta,k}(\phi,\theta)=\widetilde{c}_{\beta,k}\left\{
\begin{array}
[c]{cl}%
\cos\left(  \frac{\beta}{2}\phi\right)  T_{k}^{\frac{\beta}{2}}(\cos\theta), &
\beta\text{ even},\medskip\\
\sin\left(  \frac{\beta+1}{2}\phi\right)  T_{k}^{\frac{\beta+1}{2}}(\cos
\theta), & \beta\text{ odd}%
\end{array}
\right.
\]
Here the constant $\widetilde{c}_{\beta,k}$ is chosen in such a way that the
functions are orthonormal on the unit sphere $S^{2}$ in $\mathbb{R}^{3}$,
\[
\int_{S^{2}}S_{\beta,k}(x)S_{\widetilde{\beta},\widetilde{k}}(x)\,dx=\delta
_{\beta,\widetilde{\beta}}\,\delta_{k,\widetilde{k}}.
\]
The functions $T_{k}^{l}$ are the associated Legendre polynomials; see
\cite{hobson1965}. In \cite{gautschi2004}, \cite{zhang1996}, one can also find
recurrence formulas for the numerical evaluation of Jacobi and Legendre
polynomials and their derivatives.

The bases for the spaces $\widetilde{\mathcal{X}}_{n}$ and $\mathcal{X}_{n}$
defined in (\ref{e9a}) and (\ref{e9b}) are again, see (\ref{e9a}) and
(\ref{e9b}), defined by
\begin{align}
\widetilde{\psi}_{m,j,k}(x)  & =\left(  1-|x|^{2}\right)  \widetilde{\varphi
}_{m,j,k}(x)\label{eq1003}\\
\psi_{m,j,k}(s)  & =\widetilde{\psi}_{m,j,k}(x),\quad\quad s=\Phi
(x)\label{e1004}%
\end{align}
For the numerical implementation we can also order the bases in
lexicographical order (still using the notation $\widetilde{\psi}$ and $\psi
$), so in the following we can assume that we have bases $\{\widetilde{\psi
}_{l}\mid l=1,\ldots,N_{n}\}$ and $\{\psi_{l}\mid l=1,\ldots,N_{n}\}$ of
$\widetilde{\mathcal{X}}_{n}$ and $\mathcal{X}_{n}$. All integrals which arise
in the formulas (\ref{e50})--(\ref{e56}) for the approximate solution of
(\ref{e14}) are transformed to $\mathbb{B}_{3}$ as has been done in
(\ref{e110})--(\ref{e114}). To evaluate the resulting integrals over the unit
ball in $\mathbb{R}^{3}$ we use spherical coordinates, and a quadrature
formula $Q_{q}$
\begin{align*}
\int_{\mathbb{B}_{3}}g(x)\,dx  & =\int_{0}^{1}\int_{0}^{2\pi}\int_{0}^{\pi
}\widetilde{g}(r,\theta,\phi)r^{2}\sin(\phi)d\phi\,d\theta\,dr\\
& \approx Q_{q}[\widetilde{g}],\quad\text{where}\\
Q_{q}[\widetilde{g}]  & \equiv\sum_{i=1}^{2q}\sum_{j=1}^{q}\sum_{k=1}^{q}%
\frac{\pi}{q}\omega_{j}\nu_{k}\widetilde{g}\left(  \frac{\zeta_{k}+1}{2}%
,\frac{\pi}{2q}i,\arccos(\xi_{j})\right)
\end{align*}
Here $\widetilde{g}$ is the representation of $g$ in spherical coordinates.
The quadrature formula $Q_{q}$ uses a trapezoidal rule in the $\theta$
direction and weighted Gauss--Legendre quadrature formulas in the $\phi$
(weights $\omega_{j}$ and nodes $\arccos(\xi_{j})$) and $r$ direction (weights
$\nu_{k}$ and nodes $(\xi_{k}+1)/2$), as described in \cite{ah2010}. With the
help of this quadrature formula we can also define the numerical approximation
of $u_{0}$, see (\ref{e16}) and (\ref{e18}), by formula (\ref{e1005}).

\section{Numerical examples}%

\begin{figure}[tb]%
\centering
\includegraphics[
height=3in,
width=3.9998in
]%
{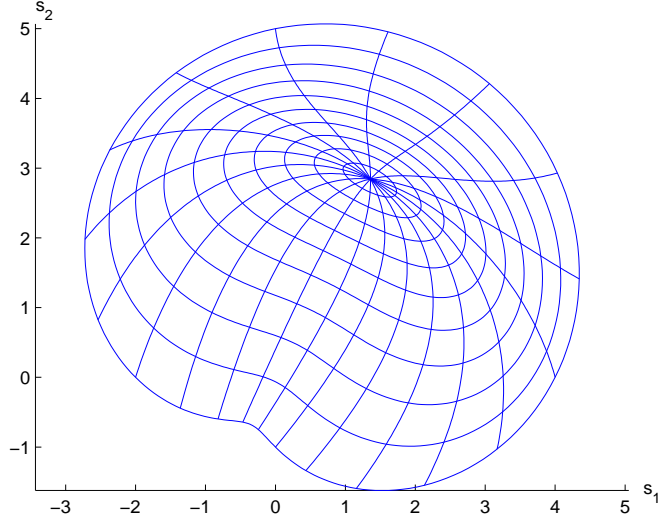}%
\caption{The region $\Omega$ associated with (\ref{e204}) and the mapping
$\Phi$}%
\label{limacon_region}%
\end{figure}
We begin with planar examples, followed by some problems on regions $\Omega$
in $\mathbb{R}^{3}$. The examples will all be for the equation%
\begin{equation}
\frac{\partial u\left(  s,t\right)  }{\partial t}=\Delta u\left(  s,t\right)
+f\left(  s,t,u\left(  s,t\right)  \right)  ,\quad\quad s\in\Omega,\quad
t\geq0.\label{e200}%
\end{equation}
To help in constructing our examples, we use%
\begin{align}
f\left(  s,t,z\right)   &  =  f_{1}\left(  s,t,z\right)  +f_{2}\left(
s,t\right)  .\label{e1006}%
\end{align}
We choose various $f_{1}$ to explore the effects of changes in the type of
nonlinearity; and $f_{2}$ is then defined to make the equation (\ref{e200})
valid for any given $u$,%
\begin{align}
f_{2}\left(  s,t\right)   &  =  \frac{\partial u\left(  s,t\right)  }{\partial
t}-\left\{  \Delta u\left(  s,t\right)  +f_{1}\left(  s,t,u\left(  s,t\right)
\right)  \right\}  ,\: s\in\Omega,\quad t\geq0.\label{e1007}%
\end{align}
In the reformulation (\ref{e114}), $A=I$ and thus%
\begin{align}
\widetilde{A}\left(  x,t,z\right)  =J\left(  x\right)  ^{-1}J\left(  x\right)
^{-\text{T}}.\label{e202}%
\end{align}

\subsection{Planar examples}

Begin with the region $\Omega$ whose boundary is a limacon. In particular,
consider the boundary%
\begin{equation}%
\begin{array}
[c]{l}%
\varphi\left(  \theta\right)  =\rho\left(  \theta\right)  \left(  \cos
\theta,\sin\theta\right)  ,\\
\rho\left(  \theta\right)  =3+\cos\theta+2\sin\theta,\quad\quad0\leq\theta
\leq2\pi.
\end{array}
\label{e204}%
\end{equation}
Using the methods of \cite{ah2011}, we obtain a mapping $\Phi:\mathbb{B}%
_{2}\rightarrow\Omega$. Each component of $\Phi$ is a polynomial of degree 3.
To illustrate the mapping we show the images in $\Omega$ of uniformly spaced
circles and radial lines in $\mathbb{B}_{2}$; see Figure \ref{limacon_region}
and note that $\Omega$ is almost convex.

As a particular example for solving (\ref{e200}), let%
\begin{align}
f_{1}\left(  s,t,z\right)   &  =e^{-z}\cos\left(  \pi t\right)  ,\medskip
\label{e206}\\
u\left(  s,t\right)   &  =\left(  1-x_{1}^{2}-x_{2}^{2}\right)  \cos\left(
t+0.05\pi s_{1}s_{2}\right) \label{e208}%
\end{align}
with $s=\Phi\left(  x\right)  $. \ For the numerical integration in
(\ref{e118}), $q=2n$ was chosen, where $n+2$ is the degree of the
approximation $\widetilde{u}_{n}$. This choice of $q$ has always been more
than adequate, and a smaller choice would often have sufficed.

To have a time interval of reasonable length, the problem was solved over
$0\leq t\leq20$, although something longer could have been chosen as well. The
error was checked at 801 points of $\Omega$, chosen as the images under $\Phi$
of 801 points distributed over $\mathbb{B}_{2}$. The graph of $u_{12}\left(
\cdot,20\right)  $ is given in Figure \ref{limacon_soln}, and the associated
error is given in Figure \ref{limacon_error}; in addition, $\left\Vert
u\left(  \cdot,20\right)  -u_{12}\left(  \cdot,20\right)  \right\Vert
_{\infty}\doteq1.94E-4$. Figure \ref{limacon_error_over_time} shows the error
norm $\left\Vert u\left(  \cdot,t\right)  -u_{12}\left(  \cdot,t\right)
\right\Vert _{\infty}$ for 200 evenly spaced values of $t$ in $\left[
0,20\right]  $. There is an oscillatory behaviour which is in keeping with
that of the solution $u$. To illustrate the spectral rate of convergence of
the method, Figure \ref{limacon_error_over_deg} gives the error as the degree
$n$ varies from $6$ to $20$. The linear behaviour of this semi-log graph
implies an exponential rate of convergence of $u_{n}$ to $u$ as a function of
$n$.
\begin{figure}[tb]%
\centering
\includegraphics[
height=3in,
width=3.9998in
]%
{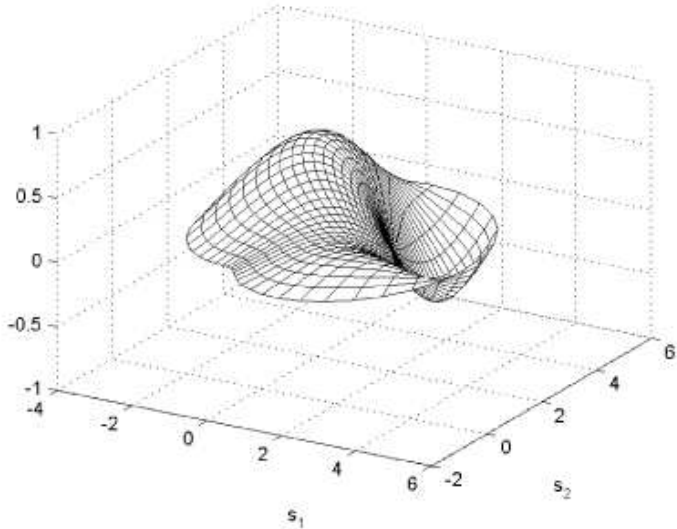}%
\caption{The approximating solution $u_{12}\left(  s,20\right)  $ for the true
solution $u\left(  s,20\right)  $ of (\ref{e208}) over $\Omega$}%
\label{limacon_soln}%
\end{figure}
\begin{figure}[tb]%
\centering
\includegraphics[
height=3in,
width=3.9998in
]%
{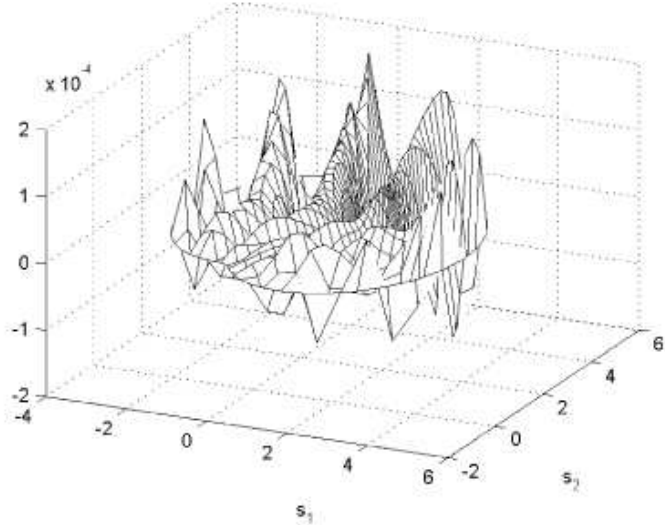}%
\caption{The error in the approximating solution $u_{12}\left(  s,20\right)  $
for the true solution $u\left(  s,20\right)  $ of (\ref{e208}) over $\Omega$}%
\label{limacon_error}%
\end{figure}
\begin{figure}[tb]%
\centering
\includegraphics[
height=3in,
width=3.9998in
]%
{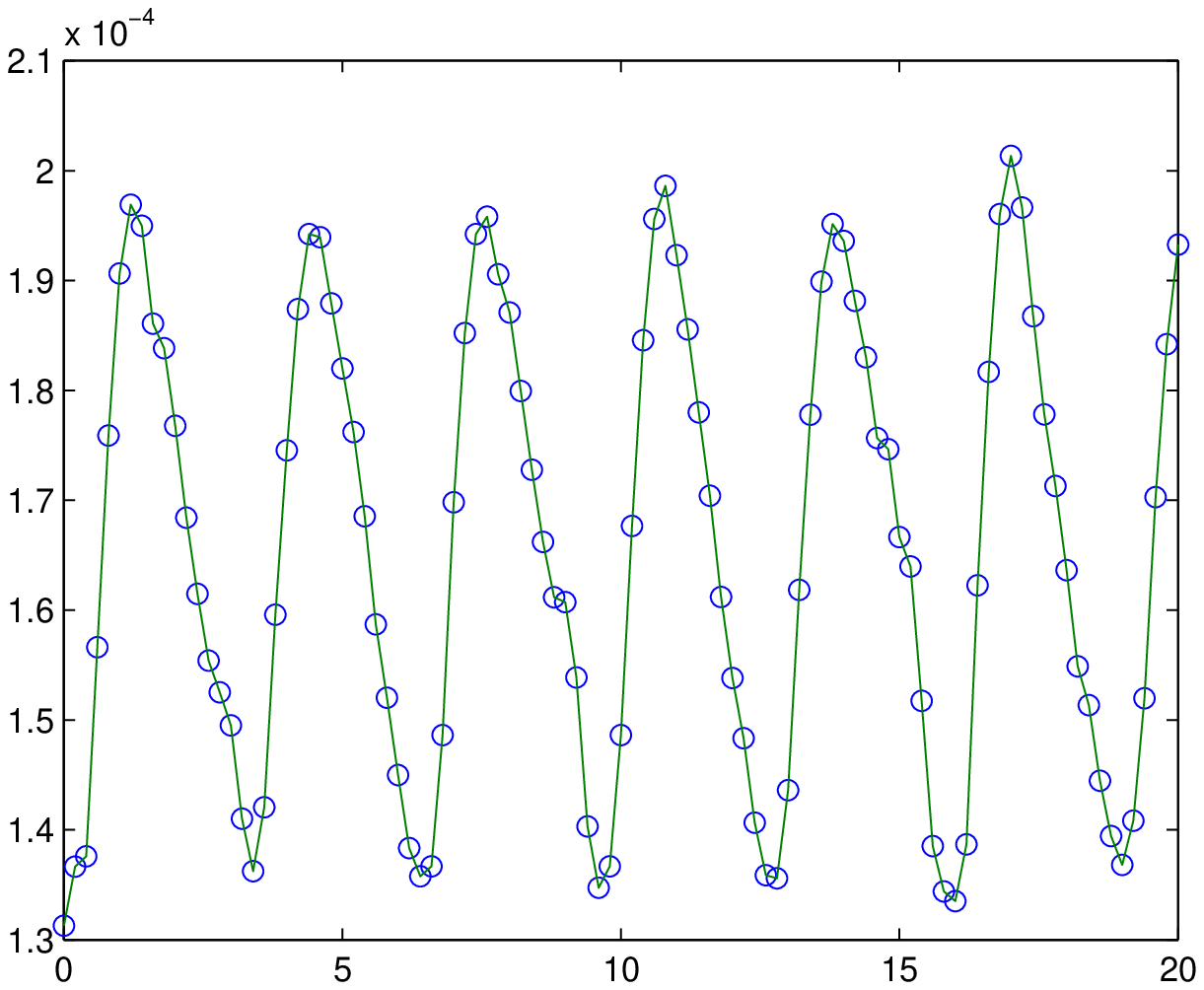}%
\caption{The error $\left\Vert u\left(  \cdot,t\right)  -u_{12}\left(
\cdot,t\right)  \right\Vert _{\infty}$ for the true solution $u\left(
s,t\right)  $ of (\ref{e208})}%
\label{limacon_error_over_time}%
\end{figure}
\begin{figure}[tb]%
\centering
\includegraphics[
height=3in,
width=3.9998in
]%
{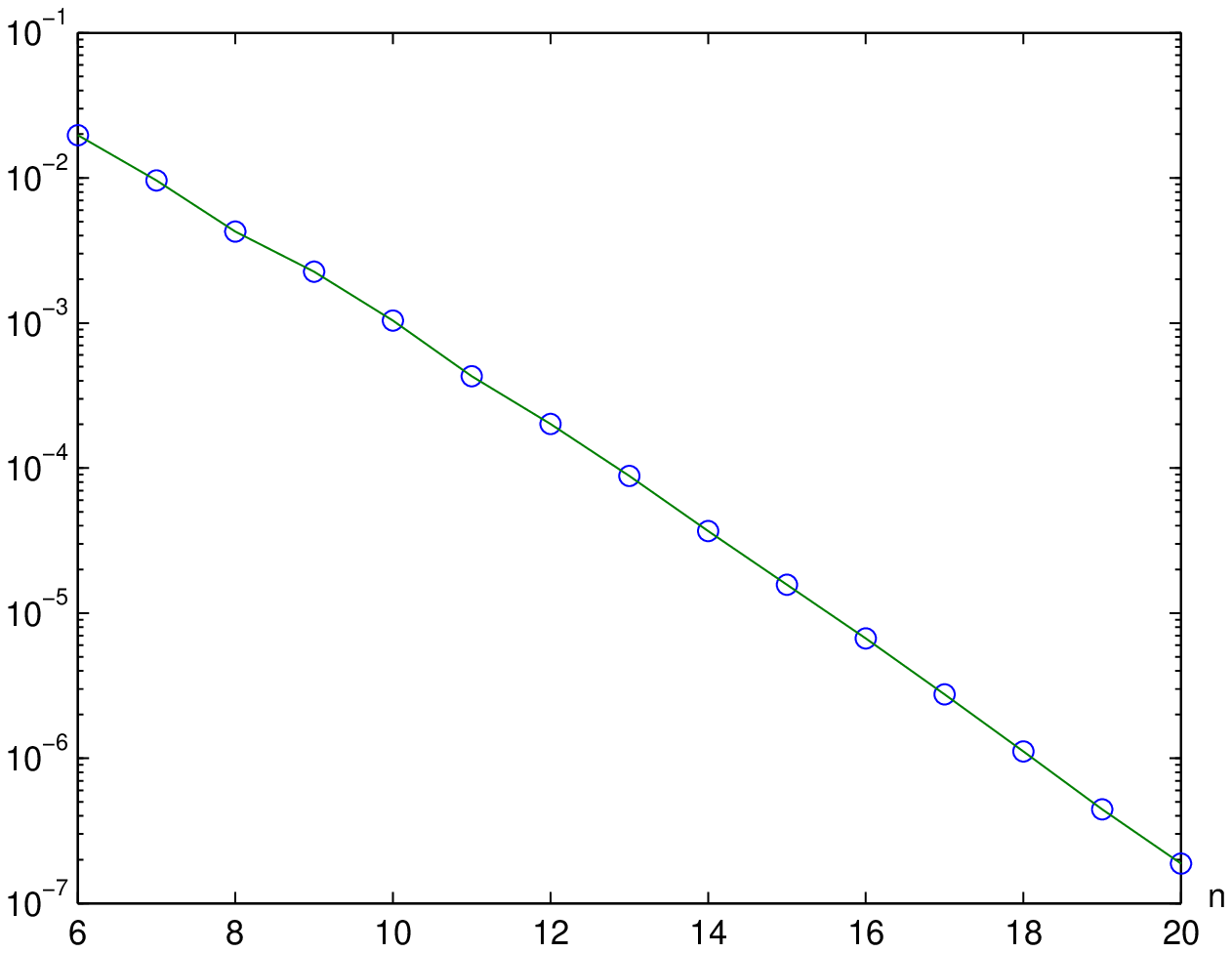}%
\caption{$n$ vs. $\max\limits_{0\leq t\leq20}\left\Vert u\left(
\cdot,t\right)  -u_{n}\left(  \cdot,t\right)  \right\Vert _{\infty}$}%
\label{limacon_error_over_deg}%
\end{figure}

An important aspect on which we have not yet commented is the conditioning of
the matrices in the system (\ref{e50}). In our use of the \textsc{Matlab}
program \texttt{ode15s}, we have written (\ref{e50}) in the form%
\begin{equation}
\mathsf{a}_{N}^{^{\prime}}\left(  t\right)  =G_{n}^{-1}B_{n}\left(
t,u_{n}\right)  \mathsf{a}_{N}\left(  t\right)  +G_{n}^{-1}\boldsymbol{f}%
_{N}\left(  t,u_{n}\right)  ,\label{e210}%
\end{equation}
The matrix $G_{n}^{-1}B_{n}\left(  t,u_{n}\right)  $ is the Jacobian matrix
for this system. \ Investigating experimentally,
\begin{equation}
\operatorname*{cond}\left(  G_{n}^{-1}B_{n}\right)  =\mathcal{O}\left(
N_{n}^{2}\right) \label{e212}%
\end{equation}
where $N_{n}$ is the number of equations in (\ref{e210}). As support for this
assertion, Figure \ref{cond_Loper} shows the graph of $\log\left(  N_{n}%
^{2}\right)  $ vs. $\log\left(  \operatorname*{cond}\left(  G_{n}^{-1}%
B_{n}\right)  \right)  $. There is a clear linear behaviour and the slope is
approximately 1, thus supporting (\ref{e212}). When $\Omega$ is the unit disk,
and $\Phi=I$, the result (\ref{e212}) is still valid experimentally.
\begin{figure}[tb]%
\centering
\includegraphics[
height=3in,
width=3.9998in
]%
{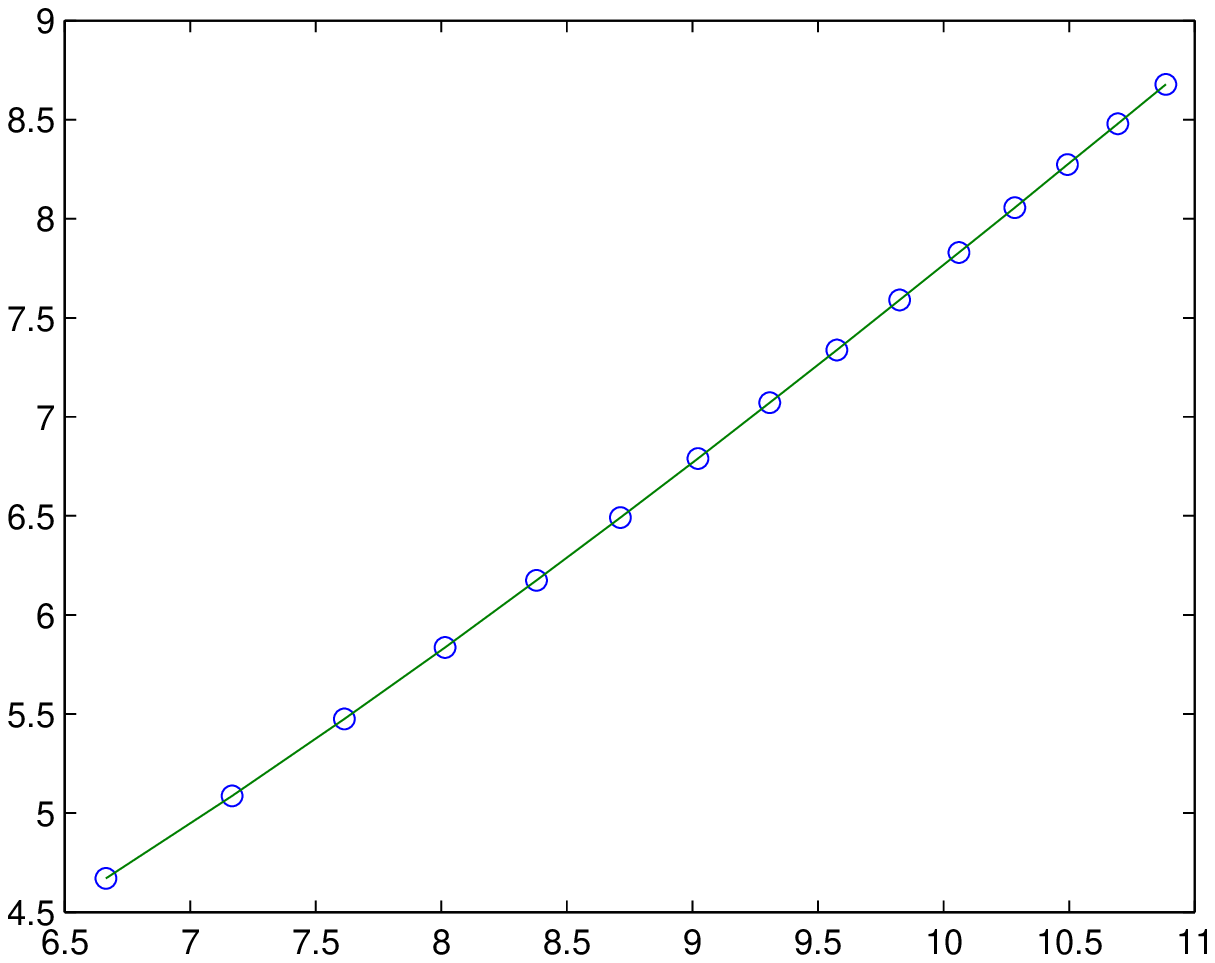}%
\caption{$\log\left(  N_{n}^{2}\right)  $ vs. $\log\left(
\operatorname*{cond}\left(  G_{n}^{-1}B_{n}\right)  \right)  $ for limacon
region}%
\label{cond_Loper}%
\end{figure}

As a second example, one for which $\Omega$ is much more nonconvex (although
still star-like), consider the region $\Omega$ with the given boundary
function%
\begin{equation}%
\begin{array}
[c]{l}%
\varphi\left(  \theta\right)  =\rho\left(  \theta\right)  \left(  \cos
\theta,\sin\theta\right)  ,\\
\rho\left(  \theta\right)  =5+\sin\theta+\sin3\theta-\cos5\theta,\quad
\quad0\leq\theta\leq2\pi.
\end{array}
\label{e214}%
\end{equation}
As before, an extension $\Phi$ to $\mathbb{B}_{2}$ is constructed using the
methods of \cite{ah2011}. The mapping $\Phi$ is a polynomial of degree 7 in
each component; and the images in $\Omega$ of uniformly spaced circles and
radial lines in $\mathbb{B}_{2}\,\ $are shown in Figure \ref{amoeba_region}.

Again, use the function $f_{1}$ of (\ref{e206}) and the solution $u$ of
(\ref{e208}). The solution $u_{20}\left(  \cdot,20\right)  $ is shown in
Figure \ref{amoeba_soln} over this new region, and $\left\Vert u\left(
\cdot,20\right)  -u_{20}\left(  \cdot,20\right)  \right\Vert _{\infty}$
$\doteq0.00136$. Figure \ref{amoeba_error_over_time} shows the error in
$u_{20}\left(  \cdot,t\right)  $ over time, and Figure
\ref{amoeba_error_over_deg} shows how the error in $u_{n}$ varies with the
degree $n$. The latter again indicates a spectral order of convergence,
although slower than that shown in Figure \ref{limacon_error_over_deg}. The
condition numbers still satisfy the empirical estimate of (\ref{e212}). \
\begin{figure}[tb]%
\centering
\includegraphics[
height=3in,
width=3.9998in
]%
{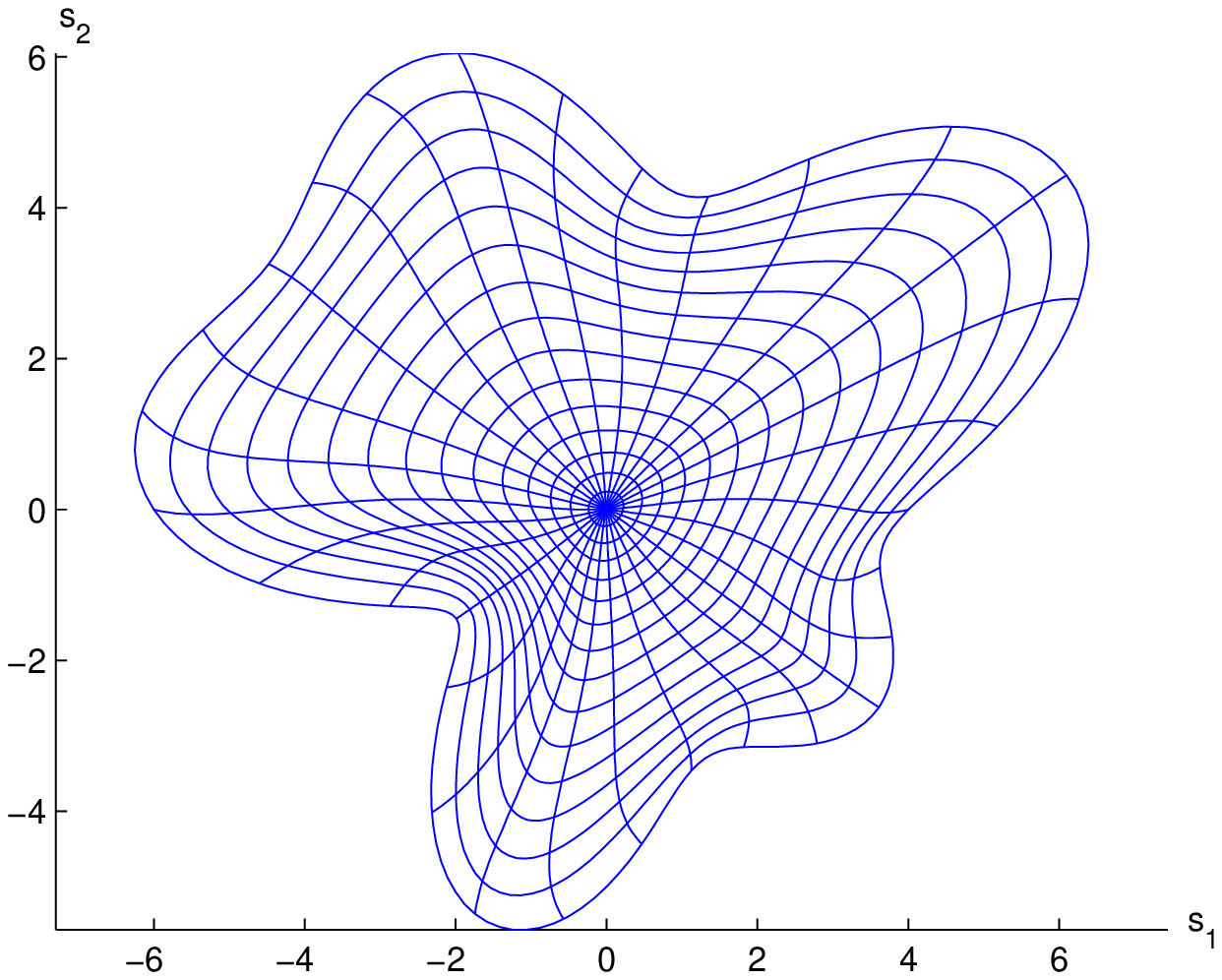}%
\caption{The region $\Omega$ associated with (\ref{e214}) and the mapping
$\Phi$}%
\label{amoeba_region}%
\end{figure}
\begin{figure}[tb]%
\centering
\includegraphics[
height=3in,
width=3.9998in
]%
{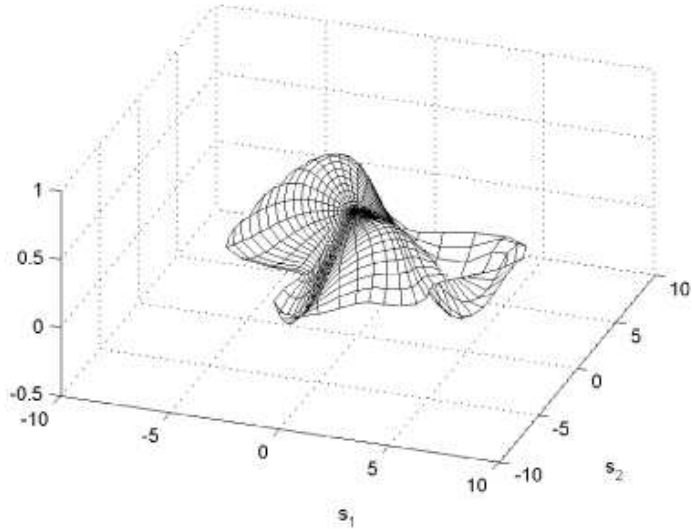}%
\caption{The approximating solution $u_{20}\left(  s,20\right)  $ for the true
solution $u\left(  s,20\right)  $ of (\ref{e208}) over the $\Omega$ of Figure
\ref{amoeba_region}}%
\label{amoeba_soln}%
\end{figure}
\begin{figure}[tb]%
\centering
\includegraphics[
height=3in,
width=3.9998in
]%
{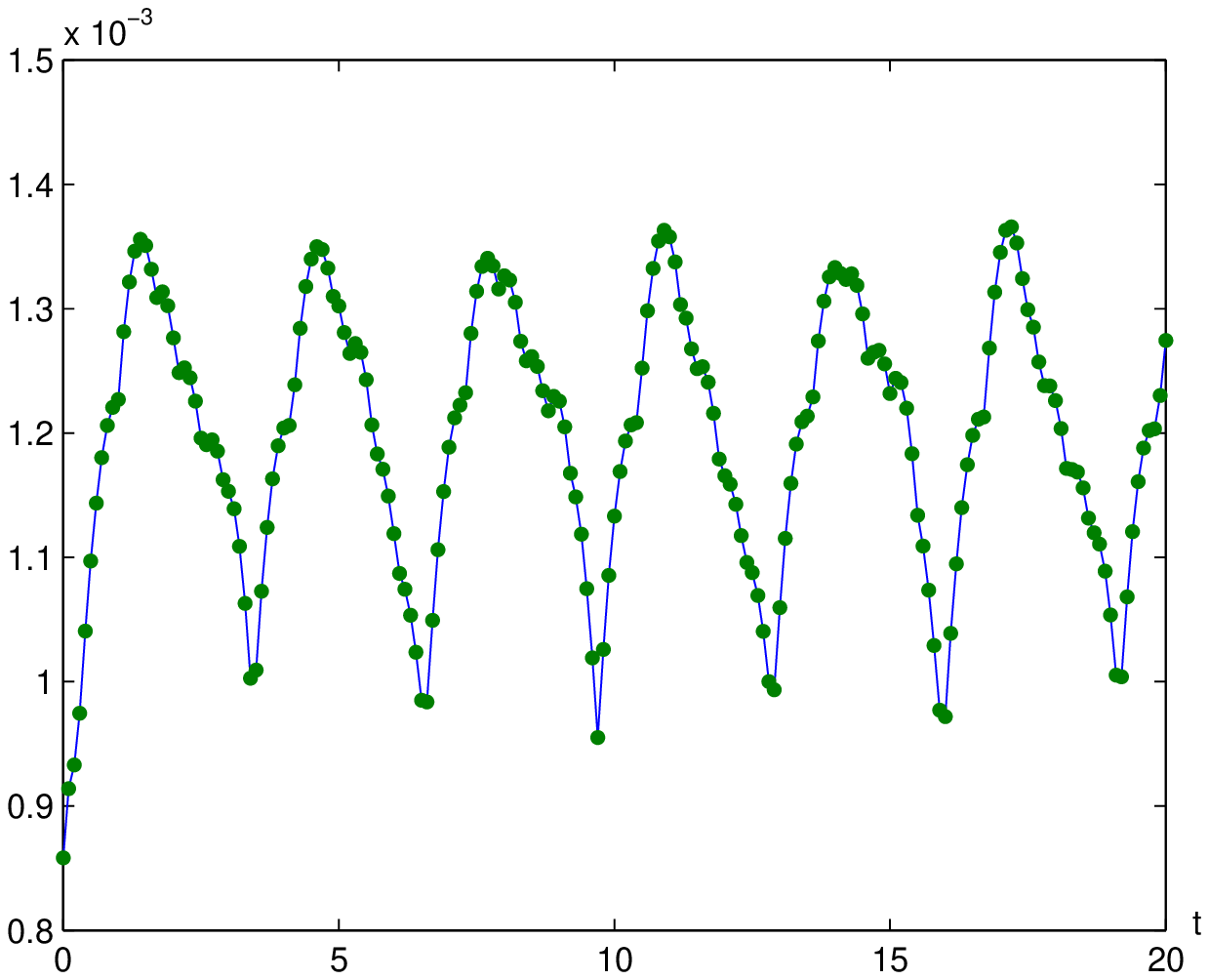}%
\caption{The error $\left\Vert u\left(  \cdot,t\right)  -u_{20}\left(
\cdot,t\right)  \right\Vert _{\infty}$ for the true solution $u\left(
s,t\right)  $ of (\ref{e208}) over the $\Omega$ of Figure \ref{amoeba_region}}%
\label{amoeba_error_over_time}%
\end{figure}
\begin{figure}[tb]%
\centering
\includegraphics[
height=3in,
width=3.9998in
]%
{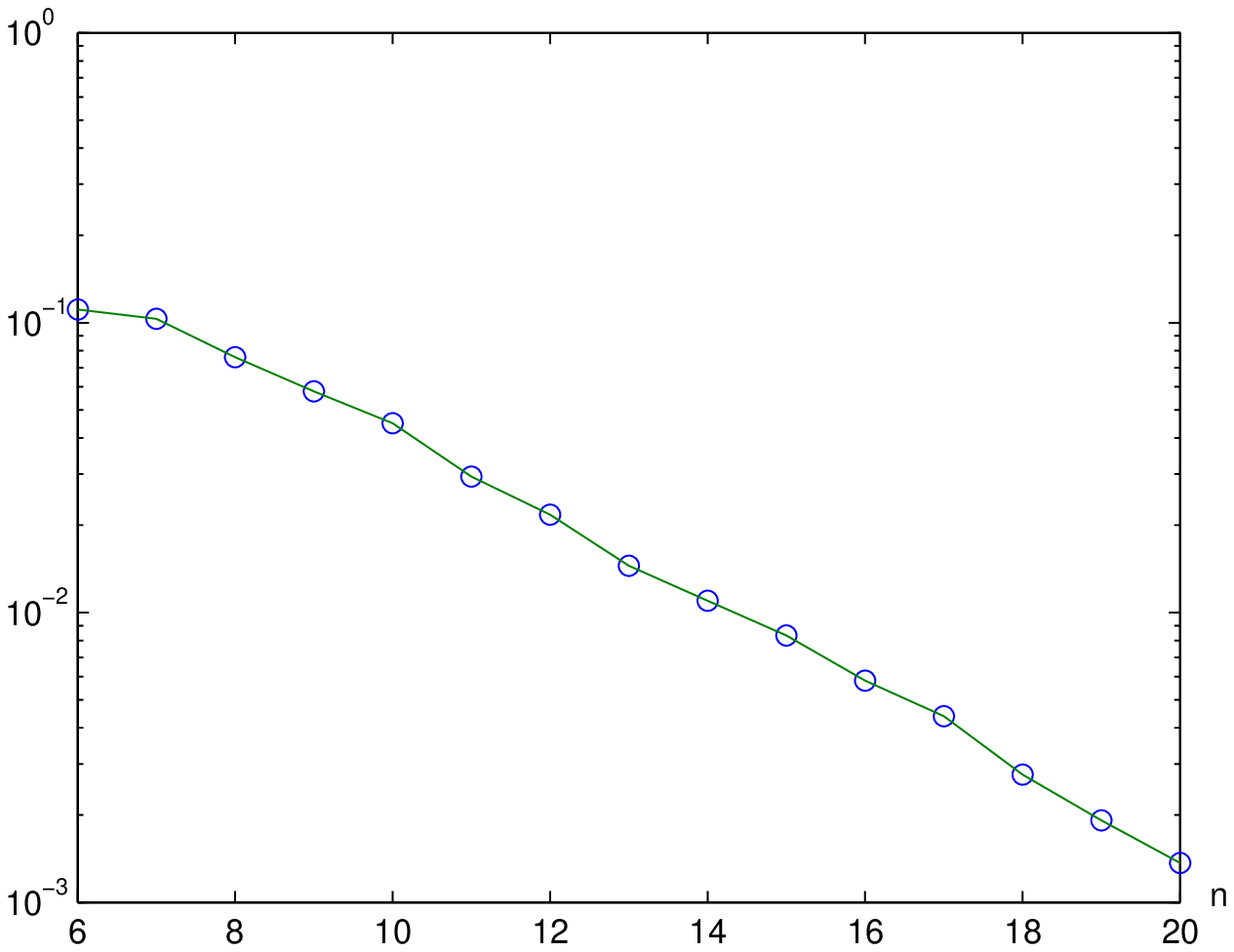}%
\caption{$n$ vs. $\max\limits_{0\leq t\leq20}\left\Vert u\left(
\cdot,t\right)  -u_{n}\left(  \cdot,t\right)  \right\Vert _{\infty}$ for the
$\Omega$ of Figure \ref{amoeba_region}}%
\label{amoeba_error_over_deg}%
\end{figure}

\subsection{ A three-dimensional example}

Here we will study one domain $\Omega$ which we investigated already in a
previous article for the purpose of analyzing the spectral method for
Dirichlet problems; see \cite{ach2008}. The domain has the advantage that the
transformation $\Phi$ is known throughout $\mathbb{B}_{3}$ and even the
inverse transformation $\Psi$ is known explicitly. The knowledge of $\Psi$ is
not necessary for the use of the spectral method but makes the construction of
an explicit solution easier. The mapping $\Phi:\overline{\mathbb{B}}%
_{3}\mapsto\overline{\Omega}$, $(s_{1},s_{2},s_{3})=\Phi(x_{1},x_{2},x_{3})$
is given by
\begin{equation}%
\begin{array}
[c]{rcl}%
s_{1} & = & x_{1}-x_{2}+ax_{1}^{2}\\
s_{2} & = & x_{1}+x_{2}\\
s_{3} & = & 2x_{3}+bx_{3}^{2}%
\end{array}
\label{e1008}%
\end{equation}
where $0<a,b<1$ are two parameters.%
\begin{figure}[tb]%
\centering
\includegraphics[
height=3in,
width=3.9998in
]%
{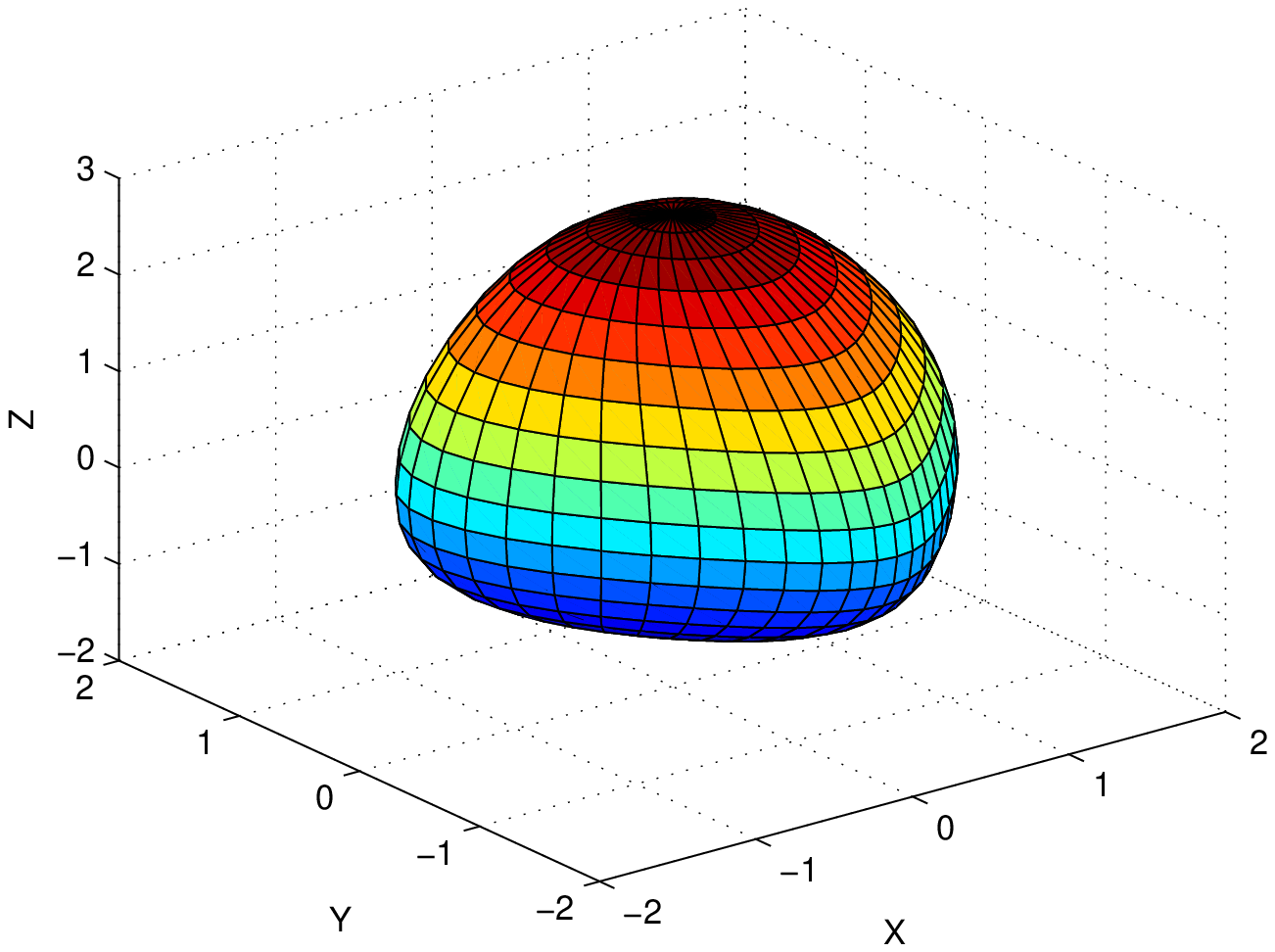}%
\caption{The surface of $\Omega$ with parameters $a=0.7$ and $b=0.9$; see
(\ref{e1008})}%
\label{SurfaceOmega3D}%
\end{figure}
\begin{figure}[tb]%
\centering
\includegraphics[
height=3in,
width=3.9998in
]%
{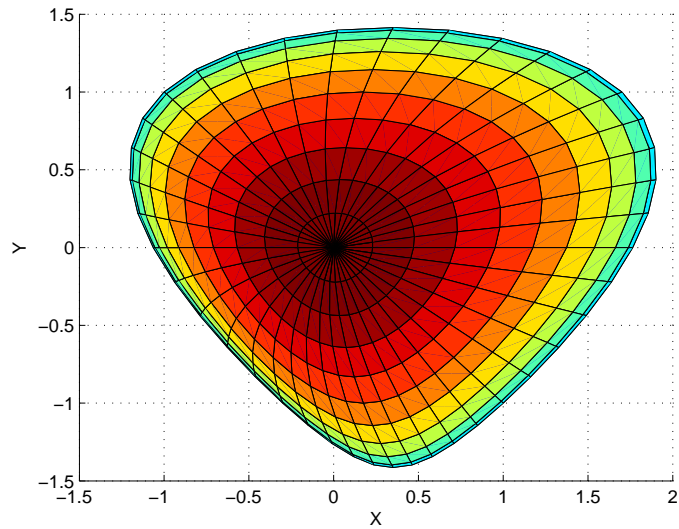}%
\caption{The surface of $\Omega$, see (\ref{e1008}), seen from the $z$-axis}%
\label{SurfaceOmega3D_V2}%
\end{figure}
Figures \ref{SurfaceOmega3D}, \ref{SurfaceOmega3D_V2} show an example of the
surface of $\Omega$ from two different angles. The inverse $\Psi
:\overline{\Omega}\mapsto\overline{\mathbb{B}}_{3}$ is given by
\begin{align*}
x_{1}  & =\frac{1}{a}\left[  -1+\sqrt{1+a(s_{1}+s_{2})}\right] \\
x_{2}  & =\frac{1}{a}\left[  as_{2}+1-\sqrt{1+a(s_{1}+s_{2})}\right] \\
x_{3}  & =\frac{1}{b}\left[  -1+\sqrt{1+bs_{3}}\right]
\end{align*}
Furthermore the Jacobian for $\Phi$ is given by%
\[
J(x)=\left(
\begin{array}
[c]{ccc}%
1+2ax_{1} & -1 & 0\\
1 & 1 & 0\\
0 & 0 & 2+2bx_{3}%
\end{array}
\right)
\]
with determinant
\[
\det(J(x))=4(1+ax_{1})(1+bx_{3}).
\]
This allows us also to calculate $\widetilde{A}$, see (\ref{e202}), directly
\[
\widetilde{A}(x)=\left(
\begin{array}
[c]{ccc}%
\dfrac{1}{2(1+ax_{1})^{2}} & \dfrac{ax_{1}}{2(1+ax_{1})^{2}} & 0\\
\dfrac{ax_{1}}{2(1+ax_{1})^{2}} & \dfrac{1+ax_{1}+2a^{2}x_{1}^{2}}%
{2(1+ax_{1})^{2}} & 0\\
0 & 0 & \dfrac{1}{4(1+bx_{3})^{2}}%
\end{array}
\right)
\]
Again we use the spectral method to solve (\ref{e200}) where $f$ is given by
(\ref{e1006}) and (\ref{e1007}). As a particular example for solving
(\ref{e200}), let
\begin{align}
f_{1}(s,t,z) &  =e^{-z}\cos(\pi t),\nonumber\\
u(s,t) &  =(1-x_{1}^{2}-x_{2}^{2}-x_{3}^{2})\cos(t+0.05\pi s_{1}s_{2}%
s_{3})\label{zz1}%
\end{align}
where $(x_{1},x_{2},x_{3})=\Psi(s_{1},s_{2},s_{3})$ with $a=0.7$ and $b=0.9$.
Numerical results are given in Figures \ref{error_3d}, \ref{n_error_3d}.
Figure \ref{cond_z54} seems to indicate that the relation (\ref{e212}) for the
condition number of the Jacobian $G_{n}^{-1}B_{n}$ is also valid in the three
dimensional case.%

\begin{figure}[tb]%
\centering
\includegraphics[
height=3in,
width=3.9998in
]%
{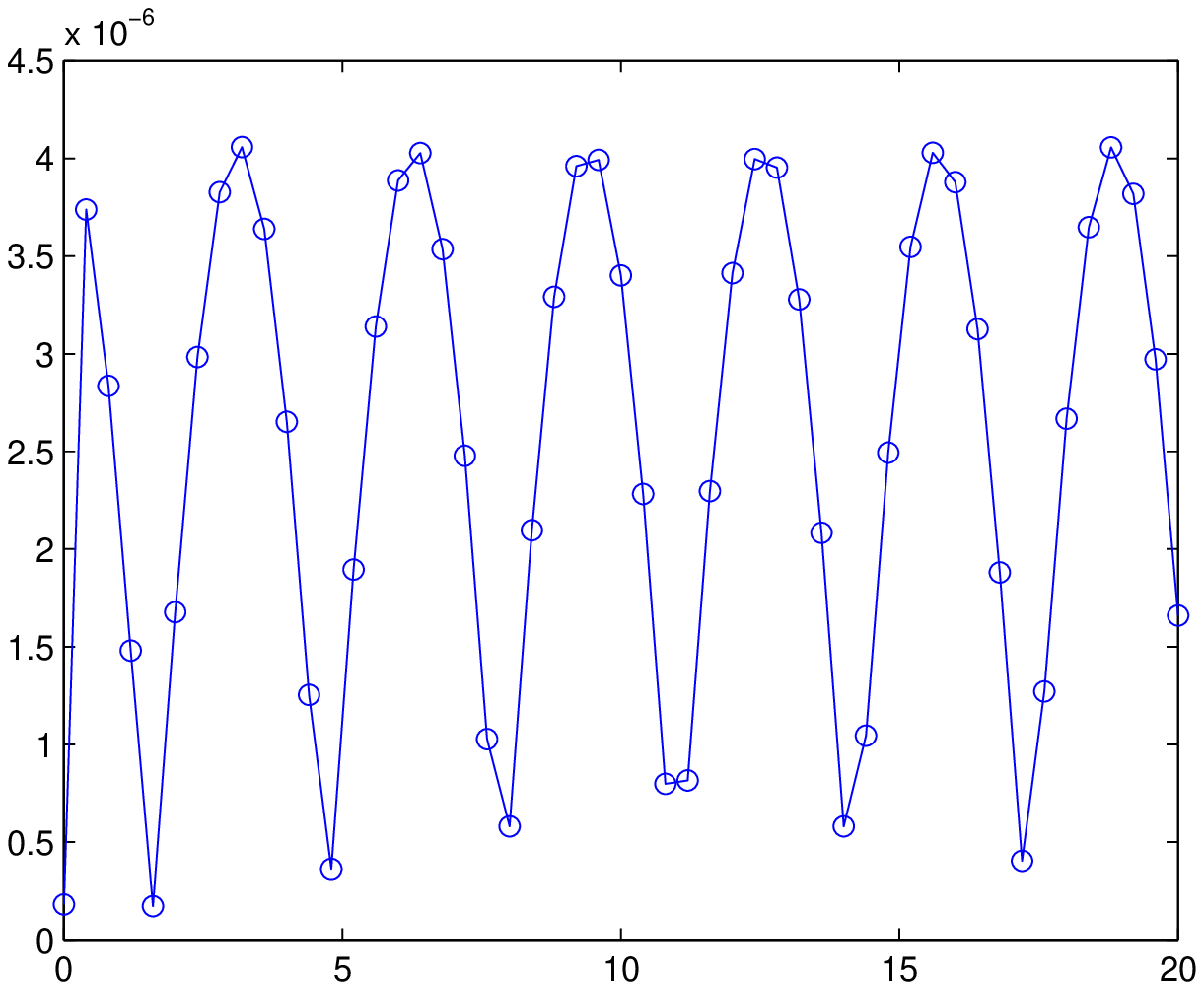}%
\caption{The error $\left\Vert u\left(  \cdot,t\right)  -u_{12}\left(
\cdot,t\right)  \right\Vert _{\infty}$ for the true solution $u\left(
s,t\right)  $ of (\ref{zz1}) over the $\Omega$ of Figure \ref{SurfaceOmega3D}}%
\label{error_3d}%
\end{figure}
%

\begin{figure}[tb]%
\centering
\includegraphics[
height=3in,
width=3.9998in
]%
{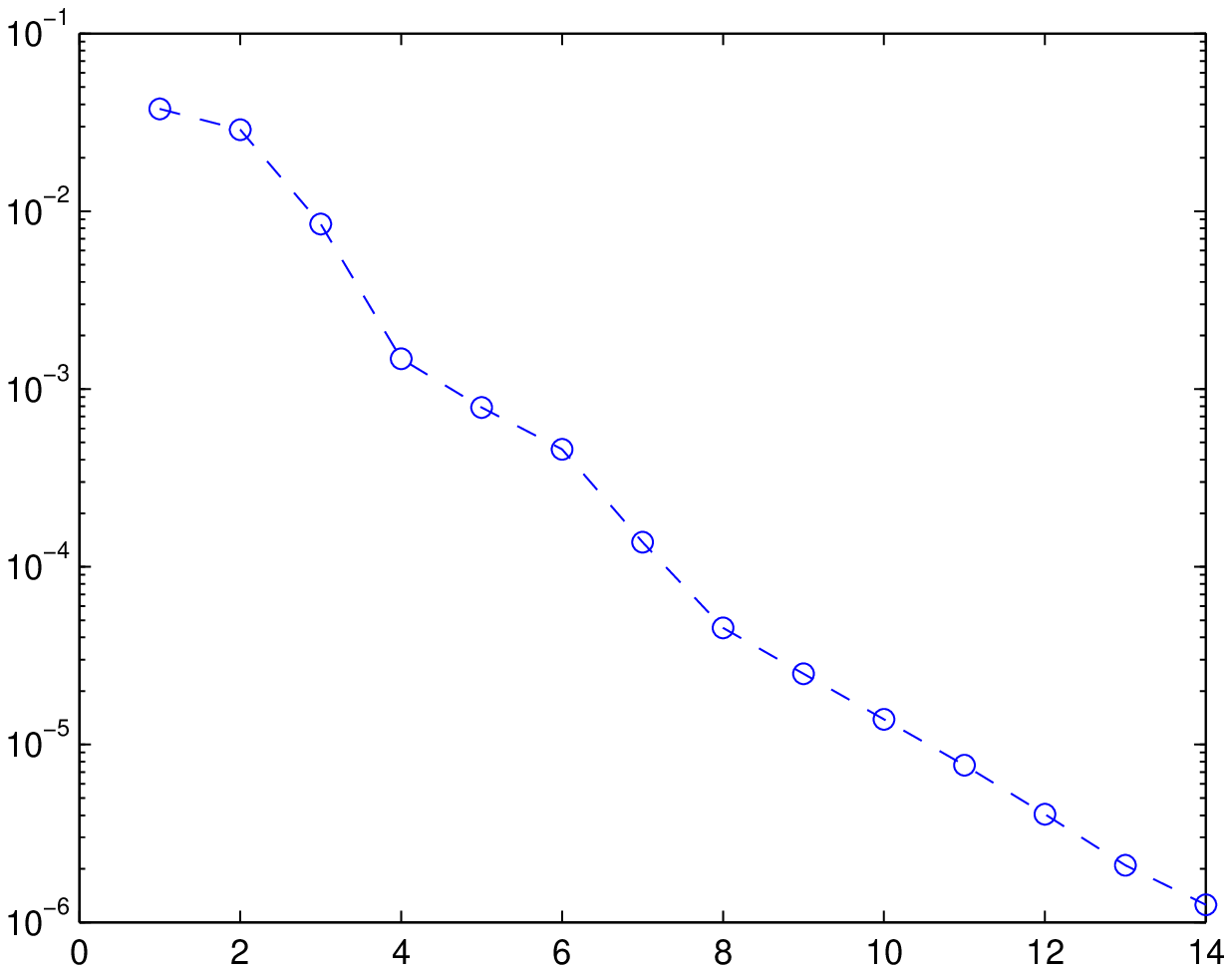}%
\caption{$n$ vs. $\max\limits_{0\leq t\leq20}\left\Vert u\left(
\cdot,t\right)  -u_{n}\left(  \cdot,t\right)  \right\Vert _{\infty}$ over the
$\Omega$ of Figure \ref{SurfaceOmega3D}}%
\label{n_error_3d}%
\end{figure}
\begin{figure}[tb]%
\centering
\includegraphics[
height=3in,
width=3.9998in
]%
{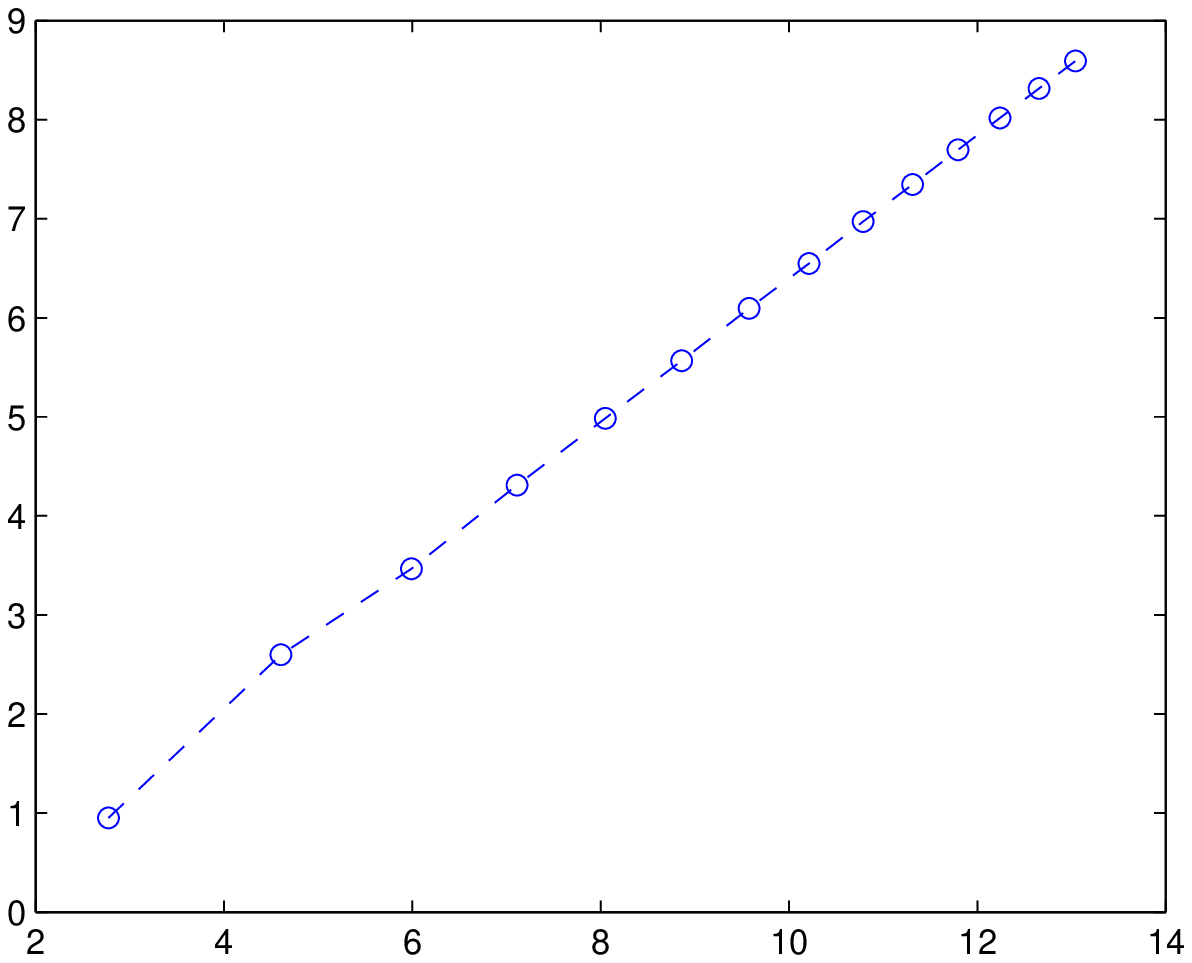}%
\caption{$\log\left(  N_{n}^{2}\right)  $ vs. $\log\left(
\operatorname*{cond}\left(  G_{n}^{-1}B_{n}\right)  \right)  $ for the
$\Omega$ of Figure \ref{SurfaceOmega3D}}%
\label{cond_z54}%
\end{figure}

\end{document}